\documentclass{article}

\usepackage{mathtools}
\usepackage{amssymb,amsthm,esint,empheq,etoolbox}
\usepackage{mathrsfs}
\usepackage{graphicx, tikz, cancel}
\usepackage{hyperref}
\usepackage[backend=biber, style=numeric, sorting=nyt,url=false, giveninits=true, isbn=false]{biblatex}
\bibliography{biblio}
\AtBeginBibliography{\footnotesize}

\newtheorem{Theorem}{Theorem}[section]
\newtheorem{Corollary}[Theorem]{Corollary}
\newtheorem{Proposition}[Theorem]{Proposition}
\newtheorem{Lemma}[Theorem]{Lemma}

\theoremstyle{definition}
\newtheorem{Definition}[Theorem]{Definition}
\newtheorem{Remark}[Theorem]{Remark}

\newcommand{\norm}[1]{\left\lVert #1 \right\rVert}
\newcommand{\abs}[1]{\left|#1\right|}

\newcommand{\mr}{\mathrm}

\newcommand{\de}{\mathrm{d}}
\newcommand{\R}{\mathbb{R}}

\newcommand{\N}{\mathbb{N}}
\newcommand{\e}{\varepsilon}
\newcommand{\del}{\delta}

\let\phi\varphi
\let\l\ell
\def\e{\varepsilon}

\newcommand{\Glim}{\Gamma\textrm{-}\lim\;}
\newcommand{\sgn}{\operatorname{sgn}}

\title{Homogenization in one-dimensional higher-order non-local models of phase transitions}

\author{Fabrizio Caragiulo, Sergio Scalabrino, Edoardo Voglino}
\date{}

\begin{document}

\maketitle
\begin{abstract}
    We study the limit behavior of Cahn--Hilliard-type functionals in which the derivative is replaced by higher-order fractional derivatives and modulated by an oscillating factor. Depending on the ratio between the oscillation scale and the interface length, we identify three different regimes and prove $\Gamma$-convergence in each regime to a suitable sharp-interface limit functional. In the extreme regimes, we prove a separation-of-scales effect that enables us to highlight the difference relative to the local models.
\end{abstract}

\section{Introduction}

The gradient theory of phase transitions \cite{MR985992, Gurtin, MR1000224, Modica} studies the behavior of double-well energies perturbed by a gradient term, following models by van Der Waals, Cahn-Hilliard, and others \cite{vdW, CH}. This analysis can be translated in the computation of a $\Gamma$-limit $F$ of functionals of the form
\begin{equation*}
F_{\e}(u)= \frac1\e\int_\Omega W(u)\,\de t+\e\int_\Omega|\nabla u|^2\de t.
\end{equation*}
If $W$ is a non-negative function with zeros exactly in $-1$ and $1$ the domain of this $\Gamma$-limit can be identified with $BV(\Omega;\{\pm 1\})$, where takes the form of a {\em sharp-interface functional} 
\[F(u)=m^1 {\rm Per}(\{u=1\},\Omega),\] proportional to the perimeter of the set $\{u=1\}$ in $\Omega$. The {\em surface tension} $m^1$ can be described by a one-dimensional  {\em optimal-profile problem} 
\begin{equation}\label{i1}
m^1= \inf\left\{\int_{-\infty}^{+\infty} (W(u)+|u'|^2)\,\de t: u\in H^1_{\rm loc}(\mathbb R), \,\lim_{t\to \pm \infty} u(t)=\pm1
\right\}
\end{equation}
(see \cite{MR985992}), which, by a scaling argument, shows that minimizers for $\e>0$ tend to develop a transition layer of scale $\e$ around a limiting minimal interface. 

When an inhomogeneity is considered, either the double-well or the gradient term, or both, depend on the spatial variable. If this dependence is highly oscillatory, the outcome is a sharp-interface functional whose effective surface tension must be described by homogenization techniques.
Usually, the inhomogeneity is modeled by introducing periodic coefficients depending on a small parameter $\delta$, which contributes to the form of the limit together with the small parameter $\e$. To describe the issues under examination, we consider the simple prototypical one-dimensional case 
\begin{equation}\label{i2}
F_{\e,\delta}(u)= \frac1\e\int_0^1 W(u)\,\de t+\e\int_0^1a\Big(\frac{t}\delta\Big)|u'|^2\de t,
\end{equation}
with the $1$-periodic function $a$ satisfying $0<\alpha\leq a\leq \beta<+\infty$ and $W$ even. 
The behavior as $\e,\delta\to 0$ is described differently in three separate regimes.  
When $\e\ll\delta$, upon assuming the function $a$ continuous (since interfaces converge to a point after scaling, this assumption removes unnecessary technicalities), the surface tension of the $\Gamma$-limit is
\begin{equation}\label{i3}
m^1_\infty= \min\big\{\sqrt{ a(\tau)}:\tau\in[0,1]\big\} m^1.
\end{equation}
 To explain formula \eqref{i3}, note that the term $a(\frac{t}\delta)$ can be considered as equal to a constant $a(\tau)$, corresponding to some $\tau\in[0,1]$, at the scale $\e$ of the internal boundary layer. Rewriting 
\begin{eqnarray}\label{i4}\nonumber
&&\hskip-1cm\frac1\e\int_0^1 W(u)\,\de t+\e\int_0^1a(\tau)|u'|^2\de t\\
&=&\sqrt{a(\tau)}\Big(
\frac{1}{\e\sqrt {a(\tau)}}\int_0^1 W(u)\,\de t+\e\sqrt{ a(\tau)}\int_0^1|u'|^2\de t\Big),
\end{eqnarray}
as a formal $\Gamma$-limit we obtain  $\sqrt{ a(\tau)}\, m^1\#(S(u))$, with $S(u)$ denoting the jump set of $u$. Note that, up to the factor $\sqrt{a(\tau)}$, this is the one-dimensional version of the sharp-interface limit when $a$ is the constant $1$. 
The surface tension $m^1_\infty$ is then obtained by optimization on $\tau$. This two-step process highlights a {\em separation-of-scales effect}: formally, we can first compute the $\Gamma$-limit for $\e\to 0$, and then for $\delta\to 0$. Conversely, if $\delta\ll\e$ then the surface tension of the $\Gamma$-limit is
\begin{equation}\label{i5}
m^1_0= \sqrt{\underline a} \,m^1, \qquad \hbox{ with } \underline a= \Big(\int_0^1\frac{1}{a(\tau)}\de\tau\Big)^{-1}
\end{equation}
denoting the {\em harmonic mean} of $a$. Again, this limit is formally obtained by first letting $\delta\to 0$, obtaining 
a first homogenized $\Gamma$-limit 
\begin{equation*}
F^{\rm hom}_{\e}(u)= \frac1\e\int_0^1 W(u)\,\de t+\e \underline a\int_0^1|u'|^2\de t,
\end{equation*}
(see e.g.~\cite{Hom_mult_int}), from which we obtain \eqref{i5} by arguing as in \eqref{i4} and  letting $\e\to 0$. The remaining {\em critical case} is when $\delta\sim\e$; that is, the inhomogeneity is at the same scale as the transition layer. In this case, there is no separation of scales, and the result depends on the ratio
$\delta/\e$, which we may assume to converge to $\lambda>0$. In this case, the formula \eqref{i1} for the surface tension must be adapted to the inhomogeneous case. Substituting $\delta$ with $\lambda\e$, a scaling argument gives the surface tension
\begin{equation*}
m^1_\lambda= \inf\Big\{\int_{-\infty}^{+\infty} \Big(W(u)+a\Big(\frac{t}\lambda\Big)|u'|^2\Big)\,\de t: u\in H^1_{\rm loc}(\mathbb R), \lim_{t\to \pm \infty} u(t)=\pm1
\Big\}.
\end{equation*}

The problem just described is a particular one-dimensional case of the model considered by Ansini et al.~in an early paper \cite{Ansini_Braides_Chiado_2003}, where also the general higher-dimensional case is treated with some technical restrictions on the parameters $\e$ and $\delta$.
Related papers are those by Dirr, Lucia, and Novaga \cite{MR2231252,MR2463798}, Cristoferi, Fonseca, and coworkers \cite{MR4633768,MR4898688,MR4014392}, and Choksi et al.~\cite{MR4404852}. In particular, the methods in the papers by Cristoferi, Fonseca et al., also treating oscillations in the double-well part, allowed one to remove the restrictions in \cite{Ansini_Braides_Chiado_2003}.  The papers \cite{MR3883579,MR4955662,MR2679586,MR4409821,MR4905549,MR4179805,MR4280521,MR4419610,MR4668552,MR4941931} deal with related issues as well. We also mention the extension to the random case, for example in \cite{MR2515785,MR4598800}.

\bigskip
In this paper, we consider a prototypical case, when the gradient term is substituted by a nonlocal Gagliardo seminorm of arbitrary order. We only treat the one-dimensional case in order not to superpose the difficulties of the fractional setting with the technicalities due to homogenization issues for functionals defined on interfaces \cite{MR1070482} (we refer to \cite{MR3748585, Braides-Oleinik-Picerni} for examples of treatment of fractional phase-transition problems in the higher-dimensional homogeneous case). We also deal with non-even $W$, for which positive jumps and negative jumps give different surface tensions, but we will omit this technicality from the introduction.

The functionals we consider are of the form
\begin{equation}\label{i12}
F_{\e,\delta}(u)= \frac1\e\int_0^1 W(u)\,\de x+\e^{2(k+s)-1}\int_0^1\int_0^1a\Big(\frac{x}\delta,\frac{y}\delta\Big)
\frac{|u^{(k)}(x)-u^{(k)}(y)|^2}{|x-y|^{1+2s}}\de x\,\de y,
\end{equation}
defined on the fractional Sobolev space $H^{k+s}((0,1))$, with $k\in\mathbb N$, $s\in (0,1)$ and $k+s>1/2$. We mention that the critical case $k=0$ and $s=\frac{1}{2}$ will be addressed in \cite{Braides-Abdilaziz}, with different methods and results, due to the logarithmic degeneracy.
The function $a$ is $1$-periodic in each variable, and it is not restrictive to suppose it is symmetric.
In the case $a=1$, the $\Gamma$-limit of such functionals has been computed by Savin and Valdinoci 
\cite{SAVIN2012479} (see also \cite{PalatucciVincini}) in the case $k=0$ and by Solci \cite{Solci_2024} in the general case, and gives a sharp-interface limit with surface tension
\begin{eqnarray}\label{i13}\nonumber
&&\hskip -1cm m^{k+s}=\inf\Big\{\int_{-\infty}^{+\infty} W(u)\,\de x+\int_{-\infty}^{+\infty}\int_{-\infty}^{+\infty} \frac{|u^{(k)}(x)-u^{(k)}(y)|^2}{|x-y|^{1+2s}}\de x\,\de y :\\[3pt] \nonumber
&&\hskip 6cm
u\in H^{k+s}_{\rm loc}(\mathbb R), \lim_{t\to \pm \infty} u(t)=\pm1
\Big\}\\
\end{eqnarray}
described by a {\em fractional optimal-profile problem}.

The asymptotic analysis of the functionals in \eqref{i12}, even if presenting the same three separate regimes, brings about significant differences from that of functionals in \eqref{i2}. By comparison with the results in \cite{Solci_2024}, the domain is $BV((0,1);\{\pm 1\})$, and we prove that the limits are still sharp-interface functionals. The behavior of the functionals $F_{\e,\delta}$ is again described by their surface tensions. 

In the {\em subcritical regime}, assuming that the function $a$ is continuous, we obtain that the surface tension of the limit functional is
\begin{equation}\label{i23}
m^{k+s}_\infty= a_{\inf}^{1/(2(k+s))} m^{k+s}, \quad\hbox{ with } a_{\inf}= \min\big\{a(\tau,\tau):\tau\in\mathbb R\big\} .
\end{equation}
%Note that this formula can be extended to non-equal periods; that is, with $a(x,y)$  $T_1$-periodic in $x$ and $T_2$-periodic in $y$. If $T_1/T_2$ is irrational then
%m^{k+s}_\infty= (\min a)^{\frac1{2(k+s)}} m^{k+s}.%
Formula \eqref{i23} highlights the additional {\em concentration} on the diagonal $x=y$, beside the {\em separation of scales} also present in the gradient case. This can be understood as a concentration of the transition layer on a scale $\e$, with $\e\ll\delta$, so that the term  $a(\frac{x}\delta,\frac{y}\delta)$ can be considered as equal to a constant $a(\tau,\tau)$ for some $\tau$. The appearance of the $2(k+s)$-th root is explained as in \eqref{i4}
by writing 
\begin{eqnarray}\label{i24}\nonumber
&&\hskip-.5cm\frac1\e\int_0^1 W(u)\,\de x+\e^{2(k+s)-1} a(\tau,\tau)\int_0^1\int_0^1\frac{|u^{(k)}(x)-u^{(k)}(y)|^2}{|x-y|^{1+2s}}\de x\,\de y\\
&&=\nonumber
\big(a(\tau,\tau)\big)^{\frac1{2(k+s)}}\Bigg(\,
\frac{1}{\e\big (a(\tau,\tau)\big)^{\frac1{2(k+s)}}}\int_0^1 W(u)\,\de t\\
&&\quad+\big(\e\big(a(\tau,\tau)\big)^{\frac1{2(k+s)}}\big)^{2(k+s)-1}\int_0^1\int_0^1\frac{|u^{(k)}(x)-u^{(k)}(y)|^2}{|x-y|^{1+2s}}\de x\,\de y\Bigg),
\end{eqnarray}
so that formally one can consider first the limit as $\e\to 0$ and then the limit as $\delta\to0$, resulting in \eqref{i23} by an optimization on $\tau$. 

Conversely, in the {\em supercritical regime}, the {\em separation of scales} entails a first limit as $\delta\to 0$ with $\e$ fixed, and a corresponding {\em homogenization} of the non-local term. 
Note that in general the $\Gamma$-limit of a family of double-integral functionals may not even be representable as a measure \cite{MR4690560, MR4201441, Mora_Corral-Tellini}. 
Nevertheless, in our case, it turns out to be a functional of the same form with the constant equal to the {\em average of $a$} in the place of the function $a$. 
The lower bound can be obtained using the lower semicontinuity of the double integral when computed on $\{(x,y): |x-y|>\eta\}$ and then letting $\eta\to 0$, while an upper bound is obtained as a pointwise limit if $u\in C^2$, and then arguing by density. The final surface tension is then given by
\begin{equation*}
m^{k+s}_0= \bar a^{1/(2(k+s))}m^{k+s}, \quad\hbox{ with } \bar a=\int_0^1\int_0^1 a(x,y)\,\de x\, \de y,
\end{equation*}
the exponent again being obtained by the scaling argument in \eqref{i24}.
We note the sharp contrast with the gradient case \eqref{i5}, where the harmonic mean substitutes the average of $a$. This discontinuity is coherent with the fact that minimum problems are singular as $s\to 0^+$ or $s\to 1^-$ as remarked in \cite{Solci_2024} using the limit analysis of Bourgain, Brezis, and Mironescu \cite{Bourgain_Brezis_Mironescu} and Maz'ya and Shaposhnikova \cite{M-Sh} and the results in \cite{Bru_Don_Sol_24}.

In the {\em critical case}, there is an analogous result, as for the gradient case, showing a competition between oscillations and transition layer at scale $\e$. 
If $\delta/\e\to\lambda\in(0,+\infty)$, then, assuming $W$ even, 
\begin{eqnarray*}\nonumber
&&\hskip-1cm m^{k+s}_\lambda= \inf\Big\{\int_{-\infty}^{+\infty}W(u)\,\de x+\int_{-\infty}^{+\infty}\int_{-\infty}^{+\infty}a\Big(\frac x\lambda,\frac y\lambda\Big)\frac{|u^{(k)}(x)-u^{(k)}(y)|^2}{|x-y|^{1+2s}}\de x\,\de y\, :\\
&&\hskip6cm\, u\in H^{k+s}_{\rm loc}(\mathbb R), \lim_{t\to \pm \infty} u(t)=\pm1
\Big\}.\nonumber
\end{eqnarray*}

Throughout the rest of the paper, we denote by $C$ a generic positive constant whose exact value may change from line to line. The constant $C$ depends on the problem parameters $k,s, \alpha_W,\beta_W,\alpha_a, \beta_a$, and possibly on additional parameters explicitly present in the statement of the results. 

\section{Results}
We introduce precise notions of double-well potential and homogenization kernel, recall the definition of fractional Sobolev spaces, and then state our main results.
%\begin{Definition}\label{def:dw}
    We say that a real continuous function $W$ is a regular \emph{double-well potential}, if
    \begin{enumerate}
     \item there exist $\alpha_W,\beta_W>0$ such that 
     \[\alpha_W(1-\abs{z})^2\leq W(z)\leq \beta_W(1-\abs{z})^2,\]
     for $\abs{z} \leq 2$;
     \item $\inf_{\{\abs{z}\ge 2\}} W(z) >0$.
\end{enumerate}
%\end{Definition}
%\begin{Definition}
%\label{def:hom kernel}
We call a measurable function $a\colon \R^2\to \R$ a \emph{homogenization kernel} if:
    \begin{enumerate}
    \item $a$ is 1-periodic in both variables,
    \item $\alpha_a =\operatorname{ess}\inf a>0$, $\beta_a=\operatorname{ess}\sup a<+\infty$.
\end{enumerate}
%\end{Definition}
A \emph{uniform family of homogenization kernels} is a family $(a_\e)_\e$ of homogenization kernels such that all the kernels of the family can be bounded by two constants $\alpha_a\equiv \alpha_{a_{\scriptstyle \e}}>0$, $\beta_a\equiv \beta_{a_{\scriptstyle\e}}<+\infty$.

\begin{Definition}
    Let $s \in (0,1)$ and $I$ be an open interval, we say that $u\in H^s(I)$ if $u\in L^2(I)$ and
    \[[u]^2_{H^s(I)} \coloneqq \int\int_{I^2} \frac{\abs{u(x)-u(y)}^2}{\abs{x-y}^{1+2s}}\,\de x\,\de y < \infty.\]
    $[\cdot]_{H^s(I)}$ is called the \emph{Gagliardo seminorm}.

    For $k\in \N$, we further say that $u\in H^{k+s}(I)$ if $u \in H^k(I)$ (the usual Sobolev space) and $u^{(k)} \in H^s(I)$. We denote $[u]_{H^{k+s}(I)} \coloneqq [u^{(k)}]_{H^s(I)}$. The space $H^{k+s}(I)$, endowed with the norm $||\cdot||_{H^{k+s}(I)}= ||\cdot||_{L^2(I)} + [\cdot]_{H^{k+s}(I)}$, is a Hilbert space. 
\end{Definition}
 
We fix a double-well potential $W$ and  {$k\in\N$ and $s\in(0,1)$ such that $k+s>\frac{1}{2}$} once for all. Given a uniform family of homogenization kernels $(a_\e)_\e$, we consider the functionals
\begin{align}
    \label{eq:def_F}\nonumber
    F^{a_{\scriptstyle \e}}_{\e }(u,I)&= \frac{1}{\e}\int_I W\big(u(x)\big)\,\de x \\
    &\qquad + \e^{2(k+s)-1} \int_I \int_{I}a_\e(x,y)\frac{|u^{(k)}(x)-u^{(k)}(y)|^2}{|x-y|^{1+2s}}\de x\,\de y,
\end{align}
defined on all $u\in H^{k+s}(I)$, $I$ an open interval. 

When $I=(0,1)$ we simply write $F^{a_{\scriptstyle \e}}_{\e}(u)$; when 
$a_\e(x,y)=a\left(\frac{x}{\del},\frac{y}{\del}\right)$
is a family of rescalings, for some $\delta\equiv \delta(\e)$, we  write $F_{\e,\delta}$ in place of $F^{a_{\scriptstyle \e}}_{\e}$.

In any case, such a family of functionals satisfies an equicoercivity property for $\e \to 0$, regardless of the choice of $\delta$, that readily follows from \cite{Solci_2024}.
\begin{Proposition}[Equicoercivity]
    \label{thm:compactness}
    Let $(a_\e)_\e$ be a uniform family of homogenization kernels and let $\{u_\e\}_{\e}$ be a family in $H^{k+s}((0,1))$ such that
 $\sup_\e F^{a_{\scriptstyle \e}}_{\e} (u_\e) < +\infty$. Then, there exist a function $u \in BV((0,1);\{\pm 1\})$ and a
 subsequence $\e_r \rightarrow 0$, independent of the homogenization kernels, such that $u_{\e_r} \rightarrow u$ in measure.
\end{Proposition}

\begin{proof}
Let us consider the functional $G_\e$ defined by
\begin{equation}
\label{funzionale Solci}
    G_\e(u)= \frac{1}{\e}\int_0^1 W(u(x))\,\de x +\e^{2(k+s)-1}\int_0^1\int_0^1\frac{|u^{(k)}(x)-u^{(k)}(y)|^2}{|x-y|^{1+2s}}\de x\,\de y.
\end{equation}
Thanks to the uniform boundedness of $a_\e$, we have for every choice of $\delta$
$$
\min\{\alpha_a,1\} G_\e(u) \leq F^{a_{\scriptstyle \e}}_{\e}(u),
$$
thus $\sup_\e G_\e(u_\e) < +\infty$, and this yields the desired subsequence $u_{\e_r}\to u$ thanks to Theorem 9 in \cite{Solci_2024}, stating the equicoercivity property for the functionals $G_\e$.
\end{proof}

In \cite{Solci_2024} it is further shown that the homogeneous functional \eqref{funzionale Solci} $\Gamma$-converges with respect to the convergence in measure:
\[
\Gamma\text{-}\lim_{\e\to 0} G_\e(u) = m\#S(u),\quad \text{for } u \in BV\big((0,1),\{\pm 1\}\big),
\]
where the jump energy density $m\equiv m^{k+s}$ is defined by \eqref{i13}, and $S(u)$ denotes the jump set of $u$, with $\# S(u)$ its cardinality. From now on, we will omit the dependence of $m$ on the regularity exponent $k+s$, to be thought of as fixed. We will also need to distinguish the directions of the jumps, and we thus introduce the following notation
\begin{equation}
\label{def: S^pm}
    S^\omega(u) = \left\{t\in S(u)\,|\, \sgn\left(u(t^+)-u(t^-)\right) = \omega\right\}, \quad \omega\in\{\pm 1\},
\end{equation}
for the sets of ascending ($\omega=1$) or descending ($\omega=-1$) jumps.

The main result of the paper can now be stated as follows.
\begin{Theorem}
\label{thm: result}
Let $W$ be a regular double-well potential, let $a$ be a homogenization kernel, and let $F_{\e,\delta}$ be defined as in \eqref{eq:def_F}, with $\delta = \delta(\e)\to 0$ as $\e$ tends to $0$. 
For every $u\in BV((0,1),\{\pm 1\})$  we have the following $\Gamma$-limits with respect to the convergence in measure, depending on the behavior of the ratio $\delta/\e$, as $\e, \delta=\delta(\e) \to 0$:
    \begin{enumerate}
        \item {\em (critical case)} For $\lambda \in (0,+\infty)$,
        \begin{equation}\label{result: case 1}
            \underset{\substack{\e\to 0 \\ \del/\e\to \lambda}}{\Glim}F_{\e,\del}(u)=\sum_{\omega\in\{\pm 1\}} m^\omega_\lambda \# S^\omega(u),
        \end{equation}
        where the jump-sets $S^\omega(u)$, $\omega\in \{\pm 1\}$, are defined  in \eqref{def: S^pm} and the \emph{transition energy} of a jump is
\begin{eqnarray*}\nonumber
&&\hskip-1.1cm m^{\omega}_\lambda= \inf\bigg\{\!\int_{-\infty}^{+\infty}\mkern-6muW(u)\,\de x+\int_{-\infty}^{+\infty}\mkern-6mu\int_{-\infty}^{+\infty}\mkern-6mu a\Big(\frac x\lambda,\frac y\lambda\Big)\frac{|u^{(k)}(x)-u^{(k)}(y)|^2}{|x-y|^{1+2s}}\de x\,\de y\,\colon\\
&&\hskip5.4cm u\in H^{k+s}_{\mathrm{loc}}(\mathbb R), \,\lim_{t\to \pm \infty} u(t)=\pm \omega
\bigg\}.\nonumber
\end{eqnarray*}

        \item {\em (supercritical case)}
        \begin{equation}
        \label{result: case 2}
            \underset{\substack{\e\to 0 \\ \del/\e\to 0}}{\Glim}F_{\e,\del}(u)=\Big(\int_0^1\int_0^1 a(x,y)\,\de x\,\de y\Big)^\frac{1}{2(k+s)}m\,\# S(u),
        \end{equation}
        where $m$ is as in \eqref{i13}.
        
        \item  {\em (subcritical case)} Assuming the homogenization kernel $a$ to be continuous,
        \begin{equation}
        \label{result: case 3}
              \underset{\substack{\e\to 0 \\ \del/\e\to +\infty}}{\Glim}F_{\e,\del}(u)= \Big(\inf_{t\in(0,1)} a(t,t)\Big)^\frac{1}{2(k+s)}m\, \#S(u),
        \end{equation}
        where $m$ is again as in \eqref{i13}.
    \end{enumerate}
\end{Theorem}
% We observe that in all three regimes the shape of the transition energy is in agreement with the previous results in the local case (see \cite{Ansini_Braides_Chiado_2003} or \cite{AnsiniBraidesPiat}), with the only exception that in \eqref{result: case 2} the coefficient depends on the homogenization kernel average, instead of its harmonic average.

To study the optimal profile problems in the definition of $m^\pm_\lambda$, and to unify the notation in all cases, it is useful to introduce the following definitions. For any homogenization kernel $a$, we define the \emph{rescaled functionals}

\begin{equation}\label{eq:def_Phi}
    \Phi^a(v, I) = \int_I W\big(v(x)\big)\,\de x +  \int_I\int_I a(x,y ) \frac{|v^{(k)}(x)-v^{(k)}(y)|^2}{|x-y|^{1+2s}}\de x\,\de y,
\end{equation}
for $v\in H^{k+s}_{\mathrm{loc}}(\R)$, and we adopt the conventions
\[\Phi^a(v)\equiv\Phi^a(v, \R), \quad  \Phi^a_T(v)\equiv  \Phi^a\big(v, (-T,T)\big).\]

\begin{Remark}
    We note that, for any family of homogenization kernels $a_\e$, fixed $t\in (0,1)$ and given $\tau>0$ small enough 
\begin{equation*}
    F^{a_{\scriptstyle \e}}_{\e}\big(u, (t-\tau, t+\tau)\big)=\Phi^{A_{\scriptstyle\e}}_{\tau/\e}(v),
\end{equation*}
if $v(x)=u(\e x+t)$ and $A_\e(x,y)=a_\e(\e x +t,\e y+t)$.
\end{Remark}
\begin{Definition}[Transition energies]
\label{def:transition_energies}
Given a homogenization kernel $a$ and $\omega\in \{\pm 1\}$, we define the \emph{transition energies} at finite and infinite length respectively as
    \begin{align*}
        m^\omega(a,T) &\coloneqq \inf \big\{\Phi^{a}(v)\;\big|\; v\in H^{k+s}_{\operatorname{loc}}(\R), \; v(x) = \omega \sgn(x)  \text{ for } \abs{x} > T\big\},\\
         m^{\omega}(a)  &\coloneqq \inf \big\{\Phi^{a}(v)\;\big|\; v\in H^{k+s}_{\operatorname{loc}}(\R), \; \lim_{x\to \pm\infty}v(x) = \pm \omega \big\}.
    \end{align*}
We also introduce the notations
\begin{equation*}
    m_{\lambda,r}^{\omega}(a)\coloneqq m^\omega\Big(a\Big(\frac{\cdot+r}{\lambda},\frac{\cdot+r}{\lambda}\Big)\Big), \quad m_\lambda^{\omega}(a)\coloneqq m^\omega\big(a(\cdot\,/\lambda,\, \cdot\,/\lambda)\big),
\end{equation*}
 for $\lambda \in (0,+\infty)$, and
 \begin{equation*}
\begin{split}
    m^\omega_0(a)=m^\omega(\bar a), \quad  &\bar a=\int_{(0,1)^2}a(x,y)\,\de x\,\de y,\\
    m^\omega_\infty(a)=m^\omega(a_{\inf}),\quad  &a_{\inf}=\inf_{t\in \R}\{a(t,t)\}.
\end{split}
\end{equation*}
The $T$-dependent versions $m_{\lambda,r}^{\omega}(a, T)$ and so on are defined analogously, by considering $m^\omega(\cdot\,, T)$ at the right-hand sides.
\end{Definition}
A few remarks are in order to shed some light on the above definitions.
\begin{Remark}\label{rmk:lambda_continuity} Let $\lambda\in (0,+\infty)$. First of all, it is clear that the definitions of $m^{\omega}_\lambda$ in the Theorem and $m^\omega_\lambda(a)$ above coincide. Moreover, at infinite length, the minimum loses its $r$ dependence, thus $m^{\omega}_{\lambda,r}(a)=m^{\omega}_\lambda(a).$
We also note that, after the change of variables $\tilde v(x)=v(\lambda x)$,
\begin{equation*}
    \begin{split}
        &\int_{-\infty}^{+\infty}\mkern-6mu W\big(v(x)\big)\,\de x +\mkern-6mu   \int_{-\infty}^{+\infty}\mkern-6mu \int_{-\infty}^{+\infty}\mkern-6mu a\Big(\frac{x}{\lambda},\frac{y}{\lambda} \Big) \frac{|v^{(k)}(x)-v^{(k)}(y)|^2}{|x-y|^{1+2s}}\de x\,\de y\\
        &=\lambda\int_{-\infty}^{+\infty}\mkern-6mu W\big(\tilde v(x)\big)\,\de x +  \lambda^{1-2(k+s)}\mkern-6mu \int_{-\infty}^{+\infty}\mkern-6mu \int_{-\infty}^{+\infty}\mkern-6mu  a(x,y) \frac{|{\tilde v}^{(k)}(x)-{\tilde v}^{(k)}(y)|^2}{|x-y|^{1+2s}}\de x\,\de y,\\
    \end{split}
\end{equation*}
from which it easily follows that $m^{\omega}_\lambda(a)$ is continuous with respect to the parameter $\lambda$.       
\end{Remark}

 \begin{Remark}
 \label{transition asymmetry} Let $a$ be any homogenization kernel and $W$ be any double-well potential. If either
 \begin{enumerate}
     \item $W$ is even, or
     \item $a$ is even in both variables,
 \end{enumerate}
 then the transition energies are independent of the direction of the jump:   $m^+(a)=m^-(a)$ and $m^+(a, T)=m^-(a, T)$ for every $T$.
 
Indeed, if $W$ is even, then $\Phi^{a}(v) = \Phi^{a}(-v)$. On the other hand, if $a$ is even, by a change of variable $\Phi^{a}(v(\,\cdot\,)) = \Phi^{a}(v(-\,\cdot\,))$. In particular, this is true for $\bar a$ and $a_{\inf}$, which are constant, and thus 
\[m_0\equiv m^\omega_0,\quad \text{and}\quad m_\infty\equiv m^\omega_\infty,\] are independent of the jump direction $\omega\in \{\pm 1\}$.
 \end{Remark}

\begin{Remark}
%\label{constant hom kernel}
 As noted in the introduction, we can relate $m_0$ and $m_\infty$ to $m\equiv m^{k+s}\equiv m(1)$, defined in \eqref{i13}, the transition energy of the homogeneous case discussed in \cite{Solci_2024}. 
    Indeed, the change of variable in Remark \ref{rmk:lambda_continuity} shows that, when $a(x,y)\equiv a\neq 0$  is constant:
    \[
    \begin{split}
        \Phi^a(v) &= \lambda \Phi^{a/\lambda^{2(k+s)}}(\tilde v)= a^{\frac{1}{2(k+s)}}\Phi^1(\tilde v),
    \end{split}
    \]
    where in the last ineqality we specialized to $\lambda=a^{1/(2(k+s))}$. Consequently,
    \[
        m_0 = (\bar a)^\frac{1}{2(k+s)}m,\qquad
        m_\infty = (a_{\inf})^\frac{1}{2(k+s)}m.
    \]
\end{Remark}

  This last remark allows us to recast even the $\Gamma$-limits \eqref{result: case 2} and \eqref{result: case 3} in the general form
  \begin{equation*}
            \underset{\substack{\e\to 0 \\ \del/\e\to \lambda}}{\Glim}F_{\e,\del}(u)=\sum_{\omega\in\{\pm 1\}} m^\omega_\lambda \# S^\omega(u), \quad \lambda\in [0,+\infty],
\end{equation*}
where, as just noted, for $\lambda\in \{0,+\infty\}$ the right-hand side is actually $m_\lambda \# S(u)$, independent of $\omega$.

\section{Preliminaries}
In this section, we recall some general properties of fractional Sobolev spaces (see \cite{Hitchhiker} and \cite{Leoni_frac_sob_sp} for a general treatment), and two useful results, inherited by comparison from \cite{Solci_2024}, concerning sequences $(u_\e)_\e$ \emph{bounded in energy}, that is, such that $\sup_\e F^{a_{\scriptstyle\e}}_{\e}(u_\e)<+\infty$, where $(a_\e)_\e$ is a fixed uniform family of homogenization kernels. The first of these results provides a bound in measure on higher derivatives, while the second bounds the number of \emph{transition intervals}, defined in \ref{def:transition_intervals}.

%The following embeddings hold.
%\begin{Proposition}\textcolor{red}{(serve?)}
 %   Let $0<s<s'<1$ and $k \in \N$ and $I$ an open interval, then
  %  \[H^{k+1}(I)\subset H^{k+s'}(I) \subset H^{k+s}(I) \subset H^k(I),\]
   % and the embeddings are continuous.
%\end{Proposition}
%\textcolor{red}{In this way, the spaces $H^{k+s}$ can be seen as interpolations.} 

The relationship between Sobolev spaces of different exponents can be described through interpolation inequalities. The first is from Theorem 1.25 and Exercise 5.9 in \cite{Leoni_frac_sob_sp}, the second is from Proposition 3 in \cite{Solci_2024} (see also \cite{Solci_2025}).
\begin{Proposition}
\label{ineq:fractional}
    Let $I \subseteq \R$ be an open interval and $s\in (0,1)$.
 There exists $C>0$ such that for all $u\in H^1(I)$ and $0< r<|I|$,
\begin{equation*}
    [u]_{H^s(I)}\leq C\big( s^{-\frac{1}{2}}r^{-s}|| u||_{L^2(I)}+(1-s)^{-\frac{1}{2}} r^{1-s} ||u'||_{L^2(I)}\big).
\end{equation*}
\end{Proposition}
\begin{proof}
    Theorem 1.25 and  Exercise 5.9 in \cite{Leoni_frac_sob_sp}.
\end{proof}

\begin{Proposition}
\label{ineq:interpolation}
Let $k \geq 1$ be a natural number and $s \in (0,1)$.
 There exists a constant $C > 0$ such that for all $I \subset \R$ bounded intervals  and all
 $u \in H^{k+s}(I)$, the following interpolation inequality holds
    \begin{equation}\label{ineq:interp_solci}
       \|u^{(\l)}\|_{L^2(I)} \leq C \left(|I|^{-\l} \|u\|_{L^2(I)} + \|u\|^{1-\theta}_{L^2(I)}[u]^{\theta}_{H^{k+s}(I)}\right),
    \end{equation}
    for all $\l= 1, \dots, k$ and for $\theta = \frac{\l}{k+s}$.
\end{Proposition}

\begin{proof}
    Proposition 3 in \cite{Solci_2024}.
\end{proof}

We will also need an improved Poincaré inequality from \cite{Solci_2024}, Lemma 4.
\begin{Proposition}
\label{ineq:mean_and_fractional} Let $I$ be a bounded interval.
   There exists a constant $C>0$ such that for every $u\in H^{s}(I)$ and any interval $J\subseteq I$, $2\abs{J} \geq \abs{I}$:
    \[\|u\|_{L^2(I)} \leq C\bigg(\frac{1}{\sqrt{\abs{I}}}\bigg|\int_Ju(x)\,\de x\bigg| + \abs{I}^{s}[u]_{H^s(I)}\bigg).\]

\end{Proposition}

We recall some classical results on compact embeddings between fractional Sobolev spaces of different exponents and Morrey's embeddings into Hölder spaces.
\begin{Theorem}\label{thm:sobolev_embd}
    Let $I$ be an interval. 
    \begin{enumerate}
        \item (Elementary embeddings) For $0<s<s'<1$ the embeddings 
        \[H^{k+1}(I)\hookrightarrow H^{k+s'}(I) \hookrightarrow H^{k+s}(I) \hookrightarrow H^k(I)\] are continuous. Moreover,  if $I$ is bounded, the embedding $H^{k+s}(I)\hookrightarrow H^k(I)$ is compact.
        \item (Morrey's embeddings) We have the following continuous embeddings.
        \begin{enumerate}
            \item If $k\geq 1$, we have that $H^{k+s}(I)\hookrightarrow C^{k-1,\frac{1}{2}}(I)$,
            in particular, there exists $C_{k,s}>0$ depending also on $I$ such that for any $\ell = 1,\dots, k$
            \[|u^{(\ell-1)}(x)-u^{(\ell-1)}(y)| \leq C_{k,s} ||u^{(\ell)}||_{L^2(I)}|x-y|^\frac{1}{2},\]
            \item if  $k=0$ and $s>\frac{1}{2}$, we have that $H^s(I)\hookrightarrow C^{0,s-\frac{1}{2}}(I)$, in particular there exists $C_{s}>0$ depending also on $I$ such that
            \[|u(x)-u(y)| \leq C_{s} [u]_{H^{s}(I)}|x-y|^{s-\frac{1}{2}}\]
        \end{enumerate} 
    \end{enumerate}
\end{Theorem}
The results are all well known. For instance, Morrey's embeddings appear as Theorem 2.8 in \cite{Leoni_frac_sob_sp}. Compactness instead can be derived by combining the previous propositions and Corollary 7.2 in \cite{Hitchhiker}.

Let us come back to the functionals $F^{a_{\scriptstyle\e}}_{\e}$. We conclude this section with two results inherited by comparison with \cite{Solci_2024}.

\begin{Lemma}[Derivative bounds]
\label{lemma:derivative_bounds}
   Let $k\geq 1$ and choose any uniform family of homogenization kernels $(a_\e)_\e$, and consider a sequence $(u_\e)_{\e}$ in $ H^{k+s}((0,1))$. 
   For each  $\eta\in(0,1)$ such that $\eta < \inf_{|z|>2}\sqrt{W(z)}$, there exists a constant $\rho > 0$ such that for any interval $I\subset (0,1)$, $|I| \geq \e \rho F^{a_{\scriptstyle\e}}_{\e}(u_\e)$,
   on which $||u_\e | -1|< \eta$, we have:
 \[
 \Big|\Big\{t \in I : |u^{(\l)}_\e(t)| < \frac{1}{\e^\l} 
\text{ for all } \l \in \{1,...,k\}\Big\}\Big|>0.
\]
\end{Lemma}
We remark that $\rho$ depends on $k$, $\alpha_W$, and the constant $C$ in the interpolation inequality of Proposition \ref{ineq:interpolation}.
\begin{proof}
    Consider the functional $G_\e$ defined in \eqref{funzionale Solci}. Then, as in Proposition \ref{thm:compactness}, $\min\{1,\alpha_a\}G_\e(u)\leq F^{a_{\scriptstyle \e}}_{\e}(u)$. Therefore, $\sup_\e G_\e(u_\e)<+\infty$, and the claim follows by the proof of Lemma $5$ in \cite{Solci_2024}.
\end{proof}
We note that in the original statement of Lemma 5 in \cite{Solci_2024}, the sequence $(u_\e)_{\e}$ belongs to $C^\infty([0,1])$, as this is a general simplifying assumption for the whole section in which the Lemma is presented. However, in the proof of this particular result, such high regularity is never used.

The next Lemma will provide a bound on the number of \emph{transition intervals}, where $u_{\e}$ jumps between two different values.
\begin{Definition}[\emph{Transition intervals}]\label{def:transition_intervals}
    Let $I$ be an open interval. Given a function $u\in C^0(\overline{I})$ and two values $\lambda_1,\lambda_2\in\R$, $\lambda_1<\lambda_2$, we say that $J=(a,b)\subset I$ is a \emph{transition interval} between $\lambda_1$ and $\lambda_2$ if $\{u(a),u(b)\} = \{\lambda_1,\lambda_2\}$.
\end{Definition}

\begin{Lemma}
    \label{lemma:transitions_number}
For any choice of a uniform family of homogenization kernels $(a_\e)_\e$, let
 $(u_\e)_{\e }$ be a sequence in   $C^\infty([0,1])$ such that $\sup_\e F^{a_{\scriptstyle\e}}_{\e}(u_\e)< +\infty$. 
 Then, for any $\lambda_1, \lambda_2\in\R$ such that $\lambda_1<\lambda_2$ and $[\lambda_1,\lambda_2]\cap \{\pm 1\} =\varnothing$ the number of transition intervals between $\lambda_1$ and $\lambda_2$ for $u_\e$ is equibounded in $\e$.
\end{Lemma}

\begin{proof}
    As in the proof of Lemma \ref{lemma:derivative_bounds}, we have that the sequence $G_\e(u_\e)$ is bounded and, thus, the claim follows by
    Proposition 10 in \cite{Solci_2024}.
\end{proof}

We conclude this section with a useful remark.
\begin{Remark}[On values far from minima]
\label{rmk:equibounded_intervals}
    We observe that the properties of the double-well potential imply that for any $\eta>0$ there exists $c_\eta>0$ such that $W(z) \geq c_\eta$ on $\{||z|-1|>\frac{\eta}{2}\}$. Then, given any  sequence $(u_\e)_\e$, 
    $\sup_\e F^{a_{\scriptstyle\e}}_{\e}(u_\e) <+\infty$, there exists a constant $C_\eta>0$ such that:
    \begin{equation*}
        \Big|\Big\{\big||u_\e|-1\big|>\frac{\eta}{2}\Big\}\Big|\leq C_\eta\e,
    \end{equation*}
    for every $\e>0$. Moreover, being the preimage of an open set under a continuous function,
    \begin{equation*}
        \Big\{\big||u_\e|-1\big|>\frac{\eta}{2}\Big\}=\bigsqcup_{n}J_n
    \end{equation*}
    is the disjoint union, at most countable, of open intervals $J_n$.
    Let $\omega_{1},\omega_{2}\in\{\pm 1\}$ be two signs. Applying Lemma \ref{lemma:transitions_number} to \[ \{\lambda_1,\lambda_2\}=\Big\{\omega_1+\omega_2\eta,\omega_1+\omega_2\frac{\eta}{2}\Big\},\] 
    with the choice $\lambda_1<\lambda_2$, we get that only a finite number of intervals, say $J_1, \ldots, J_M$, can contain a transition between $\lambda_1$ and $\lambda_2$ for any choice of signs. Moreover $M= M_\eta$ is independent from $\e$. 

    The previous considerations, upon rescaling, can be expressed in terms of the rescaled functionals $\Phi$. Let $(v_\e)_\e$ be a sequence of functions each in $H^{k+s}((-\tau/\e,\,\tau/\e ))$, where the parameter $  \tau>0$ is introduced  for later convenience, such that  
     \[\sup_\e\Phi^{\tilde a_{\scriptstyle\e}}_{\tau/\e}(v_\e)\equiv\sup_\e \Phi^{\tilde a_{\scriptstyle \e}} \left(v_\e, \left(-\frac{\tau}{\e},\frac{\tau}{\e}\right)\right)<\infty,\]
    where $\tilde a_\e(x,y)\coloneqq a_\e(\e x, \e y)$.  Then 
    \begin{equation}\label{eq:A_eta}
        \Big|\Big\{\big||v_\e|-1\big|>\frac{\eta}{2}\Big\}\Big|\leq C_\eta.
    \end{equation}
    and, again, only a finite number of disjoint intervals, say $ I_1, \ldots,  I_M\subset (-\tau/\e,\,\tau/\e)$, can contain a transition between $\lambda_1$ and $\lambda_2$ for any choice of signs. If two intervals $ I_n=(a_n,b_n)$, $ I_{n'}=(b_n,b_{n'})$ share an edge, we actually substitute them with the interior of the closure of their union, that is we consider $I_{n}= (a_n, b_{n'})$ with a small abuse of notation.
    
    This assures that, setting
    \begin{equation}
    \label{def:A_eta}
        A^\eta_{\e}\coloneqq \bigsqcup_{n=1}^M I_n,
    \end{equation} 
     both $A^\eta_\e$ and its complement $(A^\eta_\e)^{\mathsf c}$ are finite unions of (at most $M+1$) intervals, $|A^\eta_\e|\leq C_\eta$ and
        $v_{\e}$ is close to a minimum of the double-well potential on $(A^\eta_\e)^{\mathsf c}$, namely
       \[\big||v_\e|-1\big|< \eta, \quad \text{ on }(A^\eta_\e)^{\mathsf c}.\]
      Indeed, otherwise $\big||v_{\e}| -1\big|$ would assume both values $\eta$ and $\eta/2$ (note that $\big||v_{\e}| -1\big|=\eta/2$ on $\partial A^\eta_\e$), which, by construction, is not possible outside of $A^\eta_\e$.

\end{Remark}

\section{Modification lemmas} 
In this section, we will first prove two useful approximation results and an immediate corollary. Indeed, to prove the liminf inequalities in the $\Gamma$-limit in section \ref{Gamma-liminf}, it will be useful to modify a minimizing sequence to obtain better properties while increasing the energy only by an arbitrarily small amount. 
This can be challenging for non-local functionals, as any modification, no matter how localized, affects the functional over the entire domain. 

In the \emph{Flattening lemma} \ref{lemma:flattening}, we modify a function $u$ when its values are close to a minimum of the double-well potential, substituting the value with the actual minima, and estimate the energy cost of the procedure. 
One of the advantages of this flattening process is that, once a function has been made constant near the edge of its domain, we can extend it to infinity without increasing the norms of its derivatives, nor the energy of the double-well potential.

In Lemma \ref{lemma:vanishing_tails}, we prove that, indeed, for any function $v$ flattened and extended as above, the difference $\Phi^a(v, I)-\Phi^a(v)$ can be controlled up to an error that can be made arbitrarily small.

\begin{Lemma}[Flattening] \label{lemma:flattening} Let $a$ be any homogenization kernel, let \[b<c''<c'<c,\quad  c'-c''\ge 1,\] and fix $\eta\in (0,1)$. For every $u\in H^{k+s}((b,c))$ such that
\[|u(x)-1|\leq \eta,\text{ for }x\in (c'',c),\]
and for every natural $N$ there exists a $v\in H^{k+s}((b,c))$ such that
\[v(x)=u(x), \text{ for }x\in (b,c''),\qquad v(x)=1, \text{ for }x\in (c',c), \]
and, for some constants $C,C_N>0$, $C$ independent of $N$,
 
\begin{equation*}
\begin{split}
        \Phi^a\big(v, (b,c)\big) &\leq \bigg(1 +\frac{C}{N}+\frac{C_N}{(c'-c'')^{2s}}\bigg)\Phi^a\big(u, (b,c)\big).
\end{split}
\end{equation*}

\end{Lemma}

\begin{figure}
    \centering
    \includegraphics[width=0.85\linewidth]{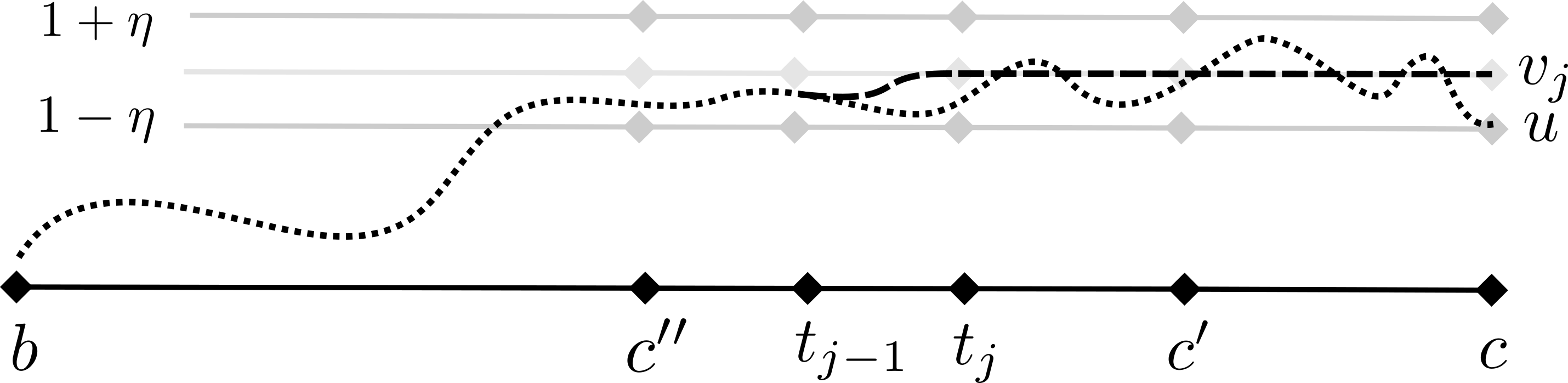}
    \caption{The dotted line represents the original function $u$, while the dashed line represents its flattening $v_j$. They differ just on the right of $t_{j-1}$.}
\end{figure}

The proof loosely follows Step 1 of the proof of the lower bound in \cite{Solci_2024}, but it is expressed in a different setting.
\begin{proof} 

For $j=1,\ldots, N$ we define $t_j\coloneqq c''+\frac{j}{N}(c'-c'')$ and consider interpolating functions $\phi_j\in C^\infty(\R)$ satisfying
    \begin{equation*}
          \phi_j(t)=
    \begin{cases}
    1 ,\quad \text{if }   t\leq t_{j-1},
    \\
    0 ,\quad \text{if } t\ge t_j.
    \end{cases}
    \end{equation*}
    We can assume $\|\phi_j^{(\l)}\|_\infty\leq C \big(\frac{N}{c'-c''}\big)^\l$, for some $C>0$ and every $\ell\leq k+1$. We define 
    \begin{equation*}
        v_j=\phi_ju+(1-\phi_j)
    \end{equation*}
    and we shall prove that one of the $v_j$ has the desired properties. 
    
    First, we focus on the local term of the functional. We can immediately notice that 
    \[\int_b^c W\big(v_j(x)\big)\,\de x = \int_b^{t_{j-1}}W\big(u(x)\big)\,\de x + \int_{t_{j-1}}^{t_j}W\big(v_j(x)\big)\,\de x.\]
    Furthermore, since $|v_j-1| \leq |u-1|<1$ in $(c'',c)\supset (t_{j-1},t_j)$, we have: 
    \begin{align*}
    \int_{t_{j-1}}^{t_j}W\big(v_j(x)\big)\,\de x &\leq \int_{t_{j-1}}^{t_j}\beta_W\big(v_j(x)-1\big)^2\,\de x \\
    &= \int_{t_{j-1}}^{t_j}\beta_W\phi_j^2(x)\big(u(x)-1\big)^2\,\de x \\ 
    &\leq \beta_W  \int_{t_{j-1}}^{t_j}\big(u(x)-1\big)^2\,\de x \\
    &\leq \frac{\beta_W}{\alpha_W}\int_{t_{j-1}}^{t_j}W\big(u(x)\big)\,\de x,
\end{align*}
hence there exists $j_*$, which will be fixed for the rest of the proof, such that
\begin{equation*}
    \int_{t_{j_*-1}}^{t_{j_*}}W\big(v_j(x)\big)\,\de x \leq \frac{1}{N} \frac{\beta_W}{\alpha_W}\int_{c''}^{c'}W\big(u(x)\big)\,\de x,
\end{equation*}
which implies
\begin{equation}
\label{flattening est loc term}
    \begin{split}
        \int_{b}^c W\big(v_{j_*}(x)\big)\,\de x \leq \int_{b}^c  W\big(u(x)\big)\,\de x + \frac{1}{N}\frac{\beta_W}{\alpha_W} \int_{c''}^{c'} W\big(u(x)\big)\,\de x.
    \end{split}
\end{equation}

We now turn to the nonlocal term of the functional.  We divide the domain of integration $(b,c)^2$ into $(b,t_{j-1})^2,\, (b,c)\times (t_{j-1},c)$ and $(t_{j-1},c) \times (b,t_{j-1})$. 
In the first case, when $x,y \in (b,t_{j-1})$, there is nothing to prove as $v_j=u$.
We can focus just on $D\coloneqq (b,c)\times (t_{j-1},c)$, because in the other case of $(x,y)\in(t_{j-1},c)\times(b,t_{j-1})$, upon switching the roles of $x$ and $y$, the reasoning is analogous.

We can compute $v_j^{(k)}(x)-v_j^{(k)}(y)$, as follows
\begin{equation*}
    \begin{split}
         v_j^{(k)}(x) - v_j^{(k)}(y)  =\sum_{\l=0}^{k}\binom{k}{\l}
     &\Big( \phi_j^{(k-\l)}(x)\big(u^{(\l)}(x)- u^{(\l)}(y)\big) \\
     &\qquad + (u-1)^{(\l)}(y)\big(\phi_j^{(k-\l)}(x)-\phi_j^{(k-\l)}(y)\big)\Big) .
    \end{split}
\end{equation*}

Using Young's inequality and that $|\phi_j|\leq 1$
\begin{equation}\label{eq:vj_estimate_decomposition}
    \begin{split}
         &\iint_{D}a(x,y)\frac{\big|v_j^{(k)}(x)-v_j^{(k)}(y)\big|^2}{|x-y|^{1+2s}}\,\de x\,\de y \\
         & \leq\Big(1+\frac{1}{N}\Big)\iint_{D} a(x,y)\frac{\big|u^{(k)}(x)-u^{(k)}(y)\big|^2}{|x-y|^{1+2s}}\,\de x\,\de y \\
         &\qquad+ C_N\beta_a \sum_{\ell=0}^{k-1}\binom{k}{\l} \underbrace{\iint_{D}\frac{\big|\phi_j^{(k-\l)}(x)\big|^2\big|u^{(\l)}(x)-  u^{(\l)}(y) \big|^2}{|x-y|^{1+2s}} \,\de x\,\de y}_{\eqqcolon G_\l}  \\ 
         &\qquad+ C_N\beta_a \sum_{\ell=0}^{k}\binom{k}{\l} \underbrace{\iint_{D}\frac{  \big|(u-1)^{(\l)}(y)\big|^2\big|\phi_j^{(k-\l)}(x)-\phi_j^{(k-\l)}(y) \big|^2}{|x-y|^{1+2s}} \,\de x\,\de y}_{\eqqcolon H_\l}.
    \end{split}
\end{equation}
The terms $H_\ell$ and $G_\ell$ can be estimated as follows:
\begin{equation}\label{ineq:Hl_Gl}
  \begin{split}
    G_\l &\leq C\bigg(\frac{N}{c'-c''}\bigg)^{2(k -\l)} \big[(u-1)^{(\l)}\big]^2_{H^s((t_{j-1},c))}, \\
      H_\l &\leq C \bigg(\frac{N}{c'-c''}\bigg)^{2(k+s-\l)}\big\| (u-1)^{(\l)}\big \|_{L^2((t_{j-1}, c))},\\
  \end{split}  
\end{equation} 
for some constant $C>0$.
Indeed, the estimate for $G_\l$ follows by the fact that $\phi_j^{(k-\ell)}(x) = 0$ if $x\leq t_{j-1}$ and $\|\phi_j^{(k-\l)}\|_\infty \leq C\big(\frac{N}{c'-c''}\big)^{k-\l}$. 

Concerning $H_\l$, it is convenient to analyze separately the integral close to and far from the diagonal. In fact, near the diagonal we use the Lipschitz continuity of $\phi^{(k-\l)}$, whereas far from the diagonal we use that $\phi^{(k-\l)}$ is bounded. We let
\begin{equation*}
    \Delta=D\cap \Big\{(x,y)\,\Big|\, |x-y|\leq \frac{c'-c''}{N} \,\Big\},
\end{equation*} 
and we estimate
\begin{equation*}
    \begin{split}
       H_\l & = \iint_{\Delta}\frac{  \big|(u-1)^{(\l)}(y)\big|^2\big|\phi_j^{(k-\l)}(x)-\phi_j^{(k-\l)}(y) \big|^2}{|x-y|^{1+2s}} \,\de x\,\de y\\
         &\qquad+ \iint_{D\setminus\Delta}\frac{  \big|(u-1)^{(\l)}(y)\big|^2\big|\phi_j^{(k-\l)}(x)-\phi_j^{(k-\l)}(y) \big|^2}{|x-y|^{1+2s}} \,\de x\,\de y\\
         & \leq C \bigg(\frac{N}{c'-c''}\bigg)^{2(k+1-\l)}  \iint_{\Delta}\frac{  \big| (u-1)^{(\l)}(y)\big|^2 }{|x-y|^{2s-1}} \,\de x\,\de y\\
         & \qquad +C \bigg(\frac{N}{c'-c''}\bigg)^{2(k-\l)} \iint_{D\setminus\Delta}   \frac{\big| (u-1)^{(\l)}(y)\big |^2}{\abs{x-y}^{1+2s}} \,\de x\,\de y\\
                  & \leq C  \bigg(\frac{N}{c'-c''}\bigg)^{2(k+1-\l)}  \big\|(u-1)^{(\l)}\big \|^2_{L^2((t_{j-1},c))} \int_{0}^{(c'-c'')/N}\xi^{1-2s} \,\de \xi\\
         & \qquad +C \bigg(\frac{N}{c'-c''}\bigg)^{2(k-\l)}    \big\|(u-1)^{(\l)}\big \|^2_{L^2((t_{j-1},c))}  \int_{(c'-c'')/N}^{+\infty}\xi^{-1-2s} \,\de \xi\\
                 & \leq C \bigg(\frac{N}{c'-c''}\bigg)^{2(k+s-\l)}  \big\|(u-1)^{(\l)}\big \|^2_{L^2((t_{j-1},c))}.
    \end{split}
\end{equation*}

 We conclude by bounding the various norms and seminorms in terms of $\Phi$. For $\l=0$ by the double well definition and $\abs{u(x)-1} < \eta$ for $x\in (t_{j-1}, c)$:
 \begin{equation}
 \label{ineq:estimate L^2 norm}
     \|u-1\|^2_{L^2((t_{j-1},c))}\leq \alpha_W^{-1}\int_{t_{j-1}}^c W (u)\,\de x \leq \alpha_W^{-1}\Phi^a(u, (b,c))
 \end{equation}
  For $\l\ge 1$ we apply the interpolation inequality (\ref{ineq:interp_solci}) to get
 \begin{equation}
 \label{ineq:estimate seminorm deriv}
     \begin{split}
         \|{u^{(\ell)}}\|^2_{L^2((t_{j-1},c))}&\leq C\big( (c-t_{j-1})^{-2\l}\| u-1\|^2_{L^2((t_{j-1},c))}\\
         &\qquad\quad+  \| u-1\|^{2-2\theta}_{L^2((t_{j-1},c))}[u]^{2\theta}_{H^{k+s}((t_{j-1},c))} \big)\\
         &\leq  C\Big(1 + \Big(\frac{1}{c'-c''}\Big)^{2\l}\Big)  \Phi^a\big(u, (b,c)\big),
     \end{split}
 \end{equation}
where we used that $c-t_{j-1} \geq c'-t_{j-1} = \frac{N-j+1}{N}(c'-c'')$. 
For $\l<k$ we apply Proposition \ref{ineq:fractional} to get
\begin{equation}
\label{ineq:estimate frac seminorm}
    \begin{split}
        [u-1] ^2_{H^{\l+s}((t_{j-1},c))}&\leq C_{N}\Big((c'-t_{j-1})^{-2s}\|(u-1)^{(\l)}\|^2_{L^2((t_{j-1},c))}\\
        &\qquad +(c'-t_{j-1})^{2(1-s)}\|(u-1)^{(\l+1)}\|^2_{L^2((t_{j-1},c))}\Big) 
    \end{split}
\end{equation}
which in turn can be estimated using \eqref{ineq:estimate L^2 norm} and \eqref{ineq:estimate seminorm deriv}.

In conclusion,  from \eqref{flattening est loc term} and the combination of \eqref{eq:vj_estimate_decomposition}, \eqref{ineq:Hl_Gl}, \eqref{ineq:estimate L^2 norm}, \eqref{ineq:estimate seminorm deriv}, and \eqref{ineq:estimate frac seminorm}, and since  $\frac{1}{c'-c''}\leq 1$, we get
 
\begin{equation*}
\begin{split}
        \Phi^a\big(v, (b,c)\big)&\leq \bigg(1+\frac{1}{N}+\frac{C_N}{(c'-c'')^{2s}}\bigg)\Phi^a\big(u, (b,c)\big)\\
        &\qquad+\frac{1}{N}\frac{\beta_W}{\alpha_W}\int_{c''}^{c'}W\big(u(x)\big)\,\de x\\
        &\leq \bigg(1 +\frac{C}{N}+\frac{C_N}{(c'-c'')^{2s}}\bigg)\Phi^a\big(u, (b,c)\big),
\end{split}
\end{equation*}
as desired.
 
\end{proof}
\begin{Remark}[Multiple flattenings]\label{rmk:multiple_flattening}
The result of Lemma \ref{lemma:flattening} also holds, with suitable modifications, for functions close to the double-well minima on multiple intervals. We spell out the details in the following two cases, for later reference. Let 
\[b<b'<b''<c''<c'< c,\quad \min\{c'-c'', b''-b'\}\ge 1,\] 
and let $u$ such that either
\[
\begin{split}
    &\text{(1)}\quad \big||u(x)|-1\big|\leq \eta, \quad \text{for }x\in (b,b'')\cup (c'',c),\\
    &\text{(2)}\quad  \big||u(x)|-1\big|\leq \eta, \quad \text{for }x\in (b',c'),
\end{split}\]
namely, in case (1), $u$ is close to some minimum on the outer subintervals, in case (2) is close to some minimum in the middle subinterval.
Then there exists $v\in H^{k+s}(b,c)$ such that, respectively, 
\[\begin{split}
    &\text{(1)}\quad\begin{cases}
         v\equiv u, &\text{on }(b'',c''), \\
        v\equiv\pm 1, &\text{on }(b,b')\cup(c',c);
    \end{cases}    \\
    &\text{(2)}\quad\begin{cases}
         v\equiv u &\text{on }(b,b')\cup (c',c), \\
          v\equiv\pm 1 &\text{on }(b'',c'').
    \end{cases}
\end{split}\]
In either case, the following inequality holds
%vecchia disuguaglianza
%
%&\leq \bigg(1+\frac{1}{N}+\frac{C_{N}}{(\min\{c'-c'', b''-b'\})^{2s}}\bigg)\Phi^a\big(u, (b,c)\big)\\
%&\qquad+\frac{1}{N}\frac{\beta_W}{\alpha_W}\int_{(b',b'')\cup (c'',c')}W\big(u(x)\big)\,\de x\\
%
\begin{equation*}
\begin{split}
        \Phi^a\big(v, (b,c)\big)
        &\leq \bigg(1+ \frac{C}{N}+\frac{C_{N}}{(\min\{c'-c'', b''-b'\})^{2s}}\bigg)\Phi^a\big(u, (b,c)\big).
\end{split}
\end{equation*}
This result follows with analogous computations after the proper choice of $\phi_j$, which this time will satisfy 
\[\|\phi_j^{(\l)}\|_\infty\leq C \Big(\frac{N}{\min\{c'-c'',b''-b'\}}\,\Big)^\l.\]
\end{Remark}

\begin{Lemma}[Vanishing tails] \label{lemma:vanishing_tails} 
Let $a$ be any homogenization kernel, $\eta$ as in Lemma \ref{lemma:derivative_bounds}, and let $v\in H^{k+s}_{\mathrm{loc}}(\R)$ be such that 
\begin{equation*}
\begin{cases}
    v(x)=\pm \operatorname{sgn}(x), &\text{ for } |x|\geq c',\\
    \big||v(x)|-1\big|\leq \eta, &\text{ for } |x|\geq c'',
\end{cases}
\end{equation*}
for some $c'>c''>1$. Then, $\Phi^a(v)<+\infty$, and, if \[c''\geq  \rho \Phi^a(v),\] where $\rho$ is constant appearing in Lemma \ref{lemma:derivative_bounds}, then, for every $T> \max\{c',3c''\}$, for $k\geq 1$:
\begin{align*}
    \Phi^a(v)&-\Phi^a_T(v) \leq C\Big(\frac{c''}{T-c'}\Big)^{2s}.
\end{align*}
    {
    While, for $k=0$ and $s>\frac{1}{2}$,
    \begin{align*}
       \Phi^a(v)&-\Phi^a_T(v) \leq C\Big(\frac{c'}{T-c'}\Big)^{2s} + \frac{C}{T^{2s-1}}. 
    \end{align*}
    }
   
\end{Lemma}
\begin{proof}  Since $v(x) = \pm 1$ for $|x|\geq c'$, we have
    \[\int_{-T}^{T}W(v)\,\de x=\int_{-c'}^{c'}W(v)\,\de x = \int_{\R}W(v)\,\de x.\]
    %This implies that $\Phi(v,\infty)-\Phi_T(v)$ only depends on the non-local term of $\Phi^a$, which is continuous with respect to the $H^{k+s}$ strong topology. Therefore we may assume $v\in C^\infty(\R)$ by density of $C^\infty$ in $H^{k+s}$. 
   This implies that $\Phi^a(v)-\Phi^a_T(v)$ only depends on the non-local term of $\Phi^a$. 
    
\paragraph{Case $k\geq 1$.} Since $v^{(k)}\equiv 0$ on $\R\setminus (-c',c')$, we have
    
\begin{equation}\label{eq:tail_stima_generale}
\begin{split}
    \Phi^a(v ) -\Phi^a_T(v )
    &\le2\beta_a\int_{\R\setminus(-T,T)}\int_{\R}\frac{| v^{(k)}(x)|^2}{|x-y|^{1+2s}}\de x\,\de y\\
    &=2\beta_a\int_{\R\setminus(-T,T)}\int_{-c'}^{c'}\frac{| v^{(k)}(x)|^2}{|x-y|^{1+2s}}\de x\,\de y\\
    &=\frac{\beta_a}{s}  \int_{-c'}^{c'}\frac{| v^{(k)}(x)|^2}{|x-T|^{2s}}+ \frac{| v^{(k)}(x)|^2}{|x+T|^{2s}}\de x \\  
    &\leq \frac{2\beta_a}{s|c'-T|^{2s}}  \int_{-c'}^{c'}|v^{(k)}(x)|^2\de x \\
    &= \frac{2\beta_a}{s(T-c')^{2s}}\|v^{(k)}\|^2_{L^2((-T,T))}.
\end{split}
\end{equation}
In particular, $\Phi^a(v)<\infty$.
    We have now to investigate the dependence of $\|v^{(k)}\|^2_{L^2((-T,T))}$ on $c',c''$.
   Set $\e=\frac{1}{2T}$, for $k\geq 2$, we apply Lemma \ref{lemma:derivative_bounds}  to the functions 
   \[u_\e(x)=v\Big(\frac{1}{\e}\Big(x-\frac{1}{2}\Big)\Big) \quad \text{for }x\in(0,1),\]
   and the intervals $J_\pm\coloneqq\e I_\pm+1/2$, where
    \[I_+=\big(c'',2c''\big), \qquad I_- =-I_+.\]
    Indeed,  $||v|-1|\leq \eta$ on $I_\pm$ and so $u_\e$ on $J_\pm$; while 
    \[|J_\pm|=\e |I_\pm|\geq \e\rho \Phi^a(v)\geq\e\rho F_\e^{\tilde{a}_\e}(u_\e),\] 
   for $\rho$ the same constant as in Lemma \ref{lemma:derivative_bounds} and $\Tilde{a}_\e(x,y) = a\left(\frac{x}{\e}-\frac{1}{2\e}, \frac{y}{\e}-\frac{1}{2\e}\right)$.
    
  Then, as a consequence of Lemma \ref{lemma:derivative_bounds}, there exist $x_+\in I_+$ and $x_- \in I_-$ such that 
\[|v^{(k-1)}(x_\pm)|\leq 1.\] 
We note that for $k=1$, $|v(x_\pm)|\leq 1+\eta$ already for any $x_\pm\in I_\pm$, and there is no need to use the Lemma.

In any case, since $ \abs{x_+-x_-} > 2c''$, we can now apply the Proposition \ref{ineq:mean_and_fractional} on the intervals 
\[  (-2c'',2c'') \supset  (x_-,x_+),\] 
obtaining:
\begin{equation}\label{eq:tail_stima_centrale}
    \begin{split}
        \|v^{(k)}\|^2_{L^2((-2c',2c'))} &\leq C\Big(\frac{1}{c''}\big|v^{(k-1)}(x_+)-v^{(k-1)}(x_-)\big|^2 + (c'')^{2s}\big[v^{(k)}\big]^2_{H^s((-2c',2c'))}\Big)\\
        &\leq C\Big(\frac{1}{c''}+(c'')^{2s}\Phi^a_{2c'}(v)\Big)\\
        &\leq C\big(1+(c'')^{2s}\Phi^a_T(v)\big).
    \end{split}
\end{equation}
On the interval $(2c'',T )$ we apply instead Proposition \ref{ineq:interpolation}, and the  \emph{Young's inequality}:
\begin{equation}\label{eq:tail_stime_laterali}
 \begin{split}
        \|v^{(k)} \|^2_{L^2((2c'',T ))} &\leq  C\Big(\Big(1+\frac{1}{(T-2c'')^{2k}}\Big)\|v\pm 1\|_{L^2((2c'',T ))}^2+[v]_{H^{k+s}((2c'',T))}^2\Big)\\
        &\leq  C(\|v\pm 1\|_{L^2((2c'',T ))}^2+[v]_{H^{k+s}((2c'',T))}^2)\\
        &\leq C  \Phi^a_T(v).
 \end{split} \end{equation}
    Together, \eqref{eq:tail_stima_centrale}, \eqref{eq:tail_stime_laterali} and an analogous inequality for $(-T,-2c'')$ imply
\begin{equation*}
    \begin{split}
        \|v^{(k)}\|^2_{L^2((-T,T))}\leq C\big(1+(c'')^{2s}\Phi^a_T(v)\big) ,
    \end{split}
\end{equation*}
which, combined with \eqref{eq:tail_stima_generale} and  $\Phi_T^a(v)\leq \Phi^a(a)$, grants the claim.

\paragraph{Case $k=0$.} Without loss of generality, we consider $v(x)= \sgn(x)$ for $|x|\geq c'$. Then,
\begin{equation}
\label{bound k=0 tails}
    \Phi^a(v ) -\Phi^a_T(v )
    \leq  2\beta_a\int_{\R\setminus(-T,T)}\int_{\R}\frac{| v(x)-\sgn(y)|^2}{|x-y|^{1+2s}}\de x\,\de y.
\end{equation}
We will estimate the contribution relative to positive $y\in (T,+\infty)$; the computations for negative $y$ are analogous.
\begin{equation}
\label{k=0 bound pos x}
    \begin{split}
\int_{T}^{+\infty}\int_{\R}\frac{| v(x)-1|^2}{|x-y|^{1+2s}}\de x\,\de y = 
&\int_{T}^{+\infty}\int_{-\infty}^{-c'}\frac{4}{|x-y|^{1+2s}}\de x\,\de y \\
&+\int_{T}^{+\infty}\int_{-c'}^{c'}\frac{| v(x)-1|^2}{|x-y|^{1+2s}}\de x\,\de y.
\end{split}
\end{equation}
By a direct computation (recall that $s>\frac{1}{2}$),
\begin{equation}
\label{bound far x}
    \int_{T}^{+\infty}\int_{-\infty}^{-c'}\frac{1}{|x-y|^{1+2s}}\de x\,\de y = \frac{1}{2s(2s-1)}\frac{1}{(T+c')^{2s-1}}\leq \frac{C}{T^{2s-1}}.
\end{equation}
On the other hand, 
\[
\begin{split}
    \int_{T}^{+\infty}\int_{-c'}^{c'}\frac{| v(x)-1|^2}{|x-y|^{1+2s}}\de x\,\de y &= \frac{1}{2s}\int_{-c'}^{c'}\frac{| v(x)-1|^2}{|T-x|^{2s}}\de x\\
    &\leq \frac{1}{2s|T-c'|^{2s}}\int_{-c'}^{c'}|v(x)-1|^2
    \,\de x
    \end{split}
\]
However, as $s>\frac{1}{2}$ by Morrey's embedding on the interval $(-T,T)$, we have
\begin{align*}
    |v(x)-1|&=|v(x)-v(c')|\\
    &\leq C [v]_{H^s((-T,T))}|x-c'|^{s-\frac{1}{2}}\\
    &\leq C (\Phi^a(v))^\frac{1}{2} |x-c'|^{s-\frac{1}{2}}\\
    &\leq C |x-c'|^{s-\frac{1}{2}},
\end{align*}
from which it follows that 
\begin{equation}
\label{last bound tails}
    \int_{T}^{+\infty}\int_{-c'}^{c'}\frac{| v(x)-1|^2}{|x-y|^{1+2s}}\de x\,\de y \leq C\frac{(c')^{2s}}{|T-c'|^{2s}}.
\end{equation}
Then we conclude by combining \eqref{bound k=0 tails}, \eqref{k=0 bound pos x}, \eqref{bound far x}, and \eqref{last bound tails}.

\end{proof}

With Lemmas \ref{lemma:flattening} and \ref{lemma:vanishing_tails} one can prove the following relation between $m^\omega(a)$ and $m^\omega(a,T)$ (see Definition \ref{def:transition_energies}). %\Edo{(riscrivere con la nuova notazione di $m$)}

\begin{Corollary}\label{coroll:min_equivalence}
For any homogenization kernel $a$ and $\omega\in \{\pm 1\}$
    \[m^{\omega}(a) = \lim_{T\rightarrow +\infty} m^{\omega}(a,T) = \inf_{T>0} m^{\omega}(a,T).\]
\end{Corollary}
 
\begin{proof} We adapt the reasoning of Proposition 14 in \cite{Solci_2024}. The inequality $m^\omega(a,T) \geq m^\omega(a)$, as well as the existence of the limit and the second equality, follow from the monotonicity of the infimum with respect to enlargements of the domain. 

 We fix a $N\in\N$ and consider $u\in H^{k+s}_{\operatorname{loc}}(\R)$ such that $\lim_{x\to \pm\infty} u(x) = \pm \omega$ and
    \begin{equation}
        \label{cor bound 1}
        \Phi^{a}(u) \leq m^\omega(a) + \frac{1}{N}.
    \end{equation}
      
    By the asymptotic behavior of $u$, there exists $c_N >0$, depending on $N$, such that $\abs{u(x)-\omega \sgn(x)}<1/N$ if $|x|> c_N$. Moreover, we can always suppose that $c_N\geq \rho \Phi^a(u)$, where $\rho$ is the constant appearing in Lemma \ref{lemma:derivative_bounds}.
    
     We use the flattening Lemma \ref{lemma:flattening}, in particular the case (1) pointed out in Remark \ref{rmk:multiple_flattening}, for some  
     \[c''> c_N,\quad c'>2c'', \quad c>c',\]
     obtaining a function $v\in H^{k+s}_{\operatorname{loc}}(\R)$ such that $v(x) = \omega \sgn(x)$ for $|x|> c'$ and such that, set $T=c$,
     \begin{equation}
     \label{cor bound 2}
        \begin{split}
            \Phi^{a}_T(v)&\leq \left(1 +  \frac{C}{N}+ \frac{C_N}{(c'')^{2s}}\right)\Phi^{a}_T(u)\\
            &\leq \left(1 +  \frac{C}{N}+ \frac{C_N}{(c'')^{2s}}\right)\Phi^{a}(u).%\\
            %&\leq \left(1 +  \frac{C}{N}+ \frac{C_N}{(c'')^{2s}}\right)\Big(m^\omega(a)+\frac{1}{N}\Big).
        \end{split}
         %+ h\int_{\{|x|\in (c'',c')\}}W(u(x))\,\de x.
     \end{equation}
  Lastly, Lemma \ref{lemma:vanishing_tails} tell us that $\Phi^{a}(v) \leq \Phi^{a}_T(v) + o_T(1)$ as $T \to \infty$ keeping $c'$ and $c''$ fixed. Then, since $v$ is an admissible competitor for $m^\omega(a,T)$, combining  \eqref{cor bound 1} and \eqref{cor bound 2}, we have
  \begin{align*}
       m^\omega(a,T)&\leq \Phi^a(v)\leq \left(1 +  \frac{C}{N}+ \frac{C_N}{(c'')^{2s}}\right)\Big(m^\omega(a)+\frac{1}{N}\Big)+o_T(1),
  \end{align*}
  and we conclude by letting $T\to \infty$, $c''\to \infty$ and $N\to \infty$ in this order.
    \end{proof}

\section{Coercivity of the rescaled energy}

In this section, we show that the rescaled energies $\Phi^a$ enjoy a coercivity property stronger than that of Proposition \ref{thm:compactness}: sequences that are bounded in energy have $k$-th derivative bounded in $L^2$, and so in $H^{s}$. 

A first consequence will be Proposition \ref{prop:positivity_of_transition_energies}: the transition energies $m^\omega(a)$ (recall Definition \ref{def:transition_energies}) are strictly positive. Hence, the limit energies of Theorem \ref{thm: result} are not trivially zero. 
Furthermore, this improved coercivity allows the extraction of subsequences strongly converging in $H^k$, which will be critical to prove the lower bound in the case $\delta\ll \e$ (see section \ref{Gamma-liminf}).

To deal with the case of $k=0$, for a given a sequence $(v_\e)_\e$ bounded in energy we will need a suitable sequence $(\sigma_\e)_{\e}$ of smooth functions, each regularizing $\sgn(v_\e)$. This function $\sigma_\e$ deals with the issue that the sequence $(v_\e)_\e$ is not bounded in $L^2$, letting us recover that $v_\e-\sigma_\e$ is bounded in $H^s$.

Throughout the rest of this section $(a_\e)_\e$ is a uniform family of homogenization kernels and $A_\e^\eta$ is defined as in Remark \ref{rmk:equibounded_intervals} for some fixed $\eta\in (0,1)$ as in Lemma \ref{lemma:derivative_bounds}.

\begin{Lemma}[Regularization of sign functions]
\label{lemma:sign_v_regularization} 
%any sequence $\delta=\delta_\e = o_\e(1)$ \textcolor{red}{(definerei la sequenza dei delta una volta per tutte ad inizio sezione)}
    Let $\tau\in (0,\frac{1}{2})$ and let 
    $(v_\e)_{\e}$ be a sequence in $H^{s}_{\operatorname{loc}}(\R)$,  $s\in (\frac{1}{2},1)$, such that \[\sup_\e\Phi^{a_{\scriptstyle \e}}_{\tau /\e}(v_\e)< \infty.\]
    
    Then, there exists a sequence $(\sigma_\e)_{\e}$ in $ C^\infty ((-\tau/\e,\,\tau/\e ))$, depending on the fixed $\eta\in (0,1)$ and $(v_\e)_\e$, such that 
    \begin{enumerate}
        \item[(i)] $\sigma_\e(x) = \sgn\big(v_\e(x)\big) $ for $ x\notin A_\e^\eta$.
        \item[(ii)] $|\sigma_\e|\leq 1$ everywhere.
        \item[(iii)] $\displaystyle \sup_\e||(\sigma_\e)'||_{L^\infty\left(\left(- \tau/\e, \, \tau/\e\right) \right)}<+\infty$.
    \end{enumerate}
\end{Lemma}

\begin{proof} We just have to prove that it is possible to define $\sigma_\e$ in $A^\eta_\e$ such that \textit{(ii)} and \textit{(iii)} hold. We recall that the sets $A^\eta_\e$ contain all transition intervals, are bounded in measure uniformly in $\e$, and they are the union of a finite number (uniformly bounded in $\e$) of intervals. Moreover, $v_\e$ is $\eta$-close to a minimum of the double-well (i.e. $\left||v_\e|-1\right|<\eta$) on the complement of $A^\eta_\e$.

Let $I$ be any connected component of $A^\eta_\e$. If $\sgn(v_\e)$ takes the same value at both edges of the interval, then we simply define $\sigma_\e$ to be constant in $I$. Otherwise, we first note that
\begin{align*}
    S&\coloneqq \sup_\e \,[v_\e]^2_{H^s((-\tau/\e,\,\tau/\e))} \leq (\alpha_a)^{-1}\Phi^{a_{\scriptstyle \e}}_{\tau/\e}(v_\e)<+\infty.
\end{align*}
Since it is always possible to define $\sigma_\e$ such that $|\sigma_\e|\leq 1$ and $|(\sigma_\e)'|\leq \frac{4}{|I|}$ on $I$, we only have to show that the length of $I$ is bounded from below independently of $\e$.
Indeed, if $I=(t,t')$,
\begin{equation*}
    |v_\e(t )-v_\e(t')| \geq   2(1-\eta),
\end{equation*}
      so that, by the Hölder regularity of $v_\e$ (see Theorem \ref{thm:sobolev_embd}):
    \[2(1-\eta)\leq CS |I|^{s-\frac{1}{2}}.\]
In the end,
    \[|(\sigma_\e)'|\leq \frac{4}{|I|}\leq C\left(\frac{S}{2(1-\eta)}\right)^\frac{2}{2s-1}\quad \text{on }I.\]
    We conclude noting that this bound can be extended to the whole $A^\eta_\e$, as it does not depend on the particular $I$ chosen.
\end{proof}

 \begin{Remark}\label{rmk:sigma_properties} Given a sequence $(v_\e)_\e$ as in the previous lemma, it is clear that, for $\e$ small enough, $\sigma_\e$ can be extended (as $\pm 1$) on the whole real line while still satisfying the three properties \textit{(i), (ii), (iii)}. With a small abuse of notation, we will still denote by $\sigma_\e$ the extensions.
 
    Observe, though, that the sequence of the norms $\|\sigma_\e\|_{L^2((-\tau/\e,\,\tau/\e))}$ is not bounded in $\e$, as the measure of the complement of $A^\eta_\e$ is diverging. 
    Therefore, neither the sequence of norms
    \[ ||\sigma_\e||_{H^s((-\tau/\e,\,\tau/\e))}\]
     is bounded in general. Still, as the next result will show, the $H^s$ seminorms are uniformly bounded.
\end{Remark}

\begin{Lemma}
\label{sigma Hs bound}
   Under the same hypotheses of Lemma \ref{lemma:sign_v_regularization}, there exists a constant $C>0$ such that for all $\xi\in \R$ 
   \begin{equation}
   \label{sigma trasl dist bound}
       \int_{-\tau/\e}^{\tau/\e} \abs{\sigma_\e(x+\xi) - \sigma_\e(x)}^2\,\de x\leq C \min\{|\xi|,\xi^2\}.
   \end{equation}
    Moreover, the seminorms $[\sigma_\e]_{H^s((-\tau/\e,\,\tau/\e))}$ are uniformly bounded.
\end{Lemma}
We notice that the constant $C$ depends only on $\eta$ and on $\sup_\e \norm{(\sigma_\e)'}_\infty$.
\begin{proof}
    We observe at first that the derivatives $(\sigma_\e)'$ are equibounded and they different from $0$ only on $A_\e^\eta$ (defined between \eqref{eq:A_eta} and \eqref{def:A_eta} in Remark \ref{rmk:equibounded_intervals}). Therefore, for $\xi>0$,
    \[
        \abs{\sigma_\e(x+\xi) - \sigma_\e(x)} \leq \int_{[x,{x+\xi}] \cap A^\eta_\e}\abs{\sigma_\e'(t)}\,\de t \leq C \big |[x,x+\xi] \cap A^\eta_\e\big|.
    \]
    Then, recalling that $A_\e^\eta = \bigsqcup_{n=1}^M I_n$, we can compute the integral for $|\xi| \geq |A_\e^\eta|$: 
    \begin{equation}
    \label{lin bound}
        \begin{split}
            \int_{-\tau/\e}^{\tau/\e} \abs{\sigma_\e(x+\xi) - \sigma_\e(x)}^2\,\de x 
            &\leq C\int_{-\tau/\e}^{\tau/\e}\big |[x,x+\xi] \cap A_\e^\eta\big|^2\,\de x\\
            & \leq C \sum_{n=1}^M \int_{-\tau/\e}^{\tau/\e} \big |[x,x+\xi] \cap I_n\big|^2\,\de x\\
            &\leq C\sum_{n=1}^M \Big( 2\int_0^{|I_n|}t^2\,\de t + (\xi- |I_n|)|I_n|^2 \Big),
        \end{split}
    \end{equation}    
    where the first term bounds the contributions of those $x$ such that $[x,x+\xi] \not\supset I_n$ and the second one corresponds to $[x,x+\xi] \supset I_n$. Since $\sum_{n=1}^M|I_n|=|A_\e^\eta|$, from \eqref{lin bound} we get
     \begin{equation}
    \label{lin bound2}
            \int_{-\tau/\e}^{\tau/\e} \abs{\sigma_\e(x+\xi) - \sigma_\e(x)}^2\,\de x \leq C\sum_{n=1}^M\left(\xi|I_n|^2-\frac{|I_n|^3}{3}\right)\leq C |A^\eta_\e|^2 \xi.
    \end{equation}      
    On the other hand,
  \begin{equation}
      \label{quad bound}
      \begin{split}
            \int_{-\tau/\e}^{\tau/\e} \abs{\sigma_\e(x+\xi) - \sigma_\e(x)}^2\,\de x 
            &\leq \int_{-\tau/\e}^{\tau/\e}\Bigg |\int_0^1(\sigma_\e)'(x+t\xi)\cdot \xi\,\de t\Bigg|^2\,\de x\\
            &\leq \xi^2\int_0^1\int_{-\tau/\e}^{\tau/\e}\big |(\sigma_\e)'(x+t\xi)\big|^2\,\de x\,\de t\\
           & \leq C|A_\e^\eta|\xi^2.
        \end{split}
  \end{equation}
    Combining \eqref{lin bound2} and \eqref{quad bound}, recalling that $|A_\e^\eta|\leq C_\eta$ uniformly on $\e$, we get \eqref{sigma trasl dist bound}, from which the bound on the $H^s$ seminorm follows easily: 
    \[
        \begin{split}
            \int_{-\tau/\e}^{\tau/\e} \int_{-\tau/\e}^{\tau/\e} \frac{\abs{\sigma_\e(y) - \sigma_\e(x)}^2}{|x-y|^{1+2s}}\, \de x\,\de y 
            &\leq \int_\R\int_{-\tau/\e}^{\tau/\e} \frac{\abs{\sigma_\e(x+\xi) - \sigma_\e(x)}^2}{\xi^{1+2s}}\, \de x\,\de \xi\\
            &\leq C\int_\R \frac{\min\{1,|\xi|\}}{|\xi|^{2s}}\, \de \xi <+\infty,
        \end{split}
    \]
    since $ \min\{1,|\xi|\}|\xi|^{-2s}$ is integrable for $2s\in (1,2)$.
\end{proof}

We can now prove the coercivity property of $\Phi^a$.

\begin{Theorem}[Coercivity in $H^{k+s}_{\operatorname{loc}}(\R)$]
    \label{coercivity}
    Let $(v_\e)_{\e}$ be a sequence in $H^{k+s}_{\operatorname{loc}}(\R)$ such that, for some  $\tau \in (0,1/2)$,
    $$\abs{v_\e (x)} \equiv 1,\quad\text{for } \abs{x}\geq \frac{\tau}{\e},$$
    and $S\coloneqq \sup_\e \Phi^{a_{\scriptstyle\e}}_{\tau/\e}(v_\e)<\infty$. 
    
    For $k \geq 1$, the sequence of the derivatives $(v_\e^{(k)})_{\e}$ is bounded in $L^2(\R)$:
    % in term of the energies $\Phi^{a_{\scriptstyle \e}}$
    there exists a constant $C>0$ depending on $S$ such that for all $\e $:
    \begin{equation}
        \label{coerc bound k>0}
        \|v_\e^{(k)}\|_{L^2(\R)}= \|v_\e^{(k)}\|_{L^2((-\tau/\e,\,\tau/\e))}\leq  C.
    \end{equation} 
    %Moreover, the $v_\e$ are equibounded and\Edo{, on any bounded interval $I$, there exists a constant $C>0$ dependent on $I$ only by its lenght $|I|$ such that
    %\[\sup_\e ||v_\e||_{H^{k+s}(I)} \leq C.\]}
    %thus, on a bounded interval $I$,
    %the norms $\|v_\e\|_{L^2(I)}$ and $\|v_\e\|_{H^{k+s}(I)}$ (recall Proposition \ref{ineq:interpolation}) are uniformly bounded.
    
    For $k =0$, recall the sequence $(\sigma_\e)_\e$ provided by Lemma \ref{lemma:sign_v_regularization} and Remark \ref{rmk:sigma_properties}. Then, the sequence $(v_\e-\sigma_\e)_{\e}$ is bounded in $L^2(\R)$:
    %in terms of the ener$\Phi^{a_{\scriptstyle \e}}$, 
    there exists a constant $C>0$, depending on $S$, such that for all $\e$:
    \begin{equation}
        \label{coerc bound k=0}
        \|v_\e-\sigma_\e\|_{L^2(\R)} =\|v_\e-\sigma_\e\|_{L^2((-\tau/\e,\,\tau/\e))}\leq C.
    \end{equation}
    
    Moreover, for each choice of $k$, the $v_\e$ are equibounded and on any bounded interval $K$ there exists a constant $C>0$ dependent only on $S$, $\sup_\e ||v_\e||_\infty$ and $|K|$ such that
    \begin{equation}
        \label{bound H^k+s loc}
        \sup_\e ||v_\e||_{H^{k+s}(K)} \leq C.
    \end{equation}
    
\end{Theorem}
\begin{proof} 
    Since $S<+\infty$, we can recall Remark \ref{rmk:equibounded_intervals}: the set $A^\eta_\e$  is open, bounded in measure uniformly in $\e$, and it is a disjoint union of at most $M$ intervals, $M$ independent of $\e$. Its complement 
    \[\left(-\frac{\tau}{\e},\frac{\tau}{\e}\right)\setminus A^\eta_\e=\bigsqcup_{n}J_n\]
    is a disjoint union of at most $M+1$ closed intervals and $||v_\e|-1|\leq \eta$ on it.
        
    To apply Lemma $\ref{lemma:derivative_bounds}$, we have to focus on those intervals $J_n$ of length larger than $2R$, for $R\coloneqq\rho S$ and $\rho$ provided by Lemma $\ref{lemma:derivative_bounds}$. We let
    \begin{equation*}
        \widetilde{A}_\e\coloneqq \Big(\bigsqcup_{\substack{
       n\colon |J_n|\ge 2R}}J_n\Big)^{\mathsf c}
    \end{equation*}
    to be the complement of the union of these large intervals. We note that $\big|\widetilde{A}_\e\big|\leq |A^\eta_\e|+ 2(M-1)R$.
    
    \paragraph{Case $k\geq 1$.} Now, let us consider any $J_n$ with $|J_n|\geq 2R$ and apply Proposition \ref{ineq:interpolation} on $v_\e-\sgn(v_\e)$, paired with \emph{Young's inequality}:
    \begin{align}\label{eq:large_intervals_estimate}
        \|v_\e^{(k)}\|_{L^2(J_n)}^2 &\leq C\left((1 +R^{-2k})\norm{v_\e-\,\sgn(v_\e)}^2_{L^2(J_n)} + [v_\e]_{H^{k+s}(J_n)}^2\right)\\
        & \leq C\Phi^{a_{\scriptstyle \e}}(v_\e,J_n),\nonumber
    \end{align}
    where the estimate of $\|v_\e-\,\sgn(v_\e)\|^2_{L^2(J_n)}$ with $\Phi^{a_{\scriptstyle \e}}(v_\e,J_n)$ is analogous to \eqref{ineq:estimate L^2 norm}. 
    
    Let now $I$ be a connected component of $\widetilde{A}_\e$, say $I=(b,c)$. For $k\geq 2$, we can apply Lemma \ref{lemma:derivative_bounds} to
    \[I_-=(b-2R,b-R),\qquad I_+=(c+R,C+2R),\]
    since on both intervals $\abs{\abs{v_\e}-1}\leq \eta$. Thus we find
    \[x_- \in (b-2R,b-R),\quad x_+\in (c+R,c+2R),\]
    such that $|v_\e^{(k-1)}(x_\pm)|\leq 1$. While for $k=1$, we already have $|v(x_\pm)|\leq 1+\eta$ for any $x_\pm \in I_\pm$.
    Moreover, since 
    \[ x_+-x_- \geq |I|+2R \geq\frac{1}{2}\big|I+(-2R,2R)\big|,\]
    we can also apply Proposition \ref{ineq:mean_and_fractional} to $(x_-,x_+)\subset I+(-2R,2R)$ and get, using Young's inequality again:
    \[
       \begin{split}
            \|v_\e^{(k)}\|_{L^2(I)}^2&\leq \|v_\e^{(k)}\|^2_{L^2(I+(-2R,2R))}\\
            &\leq C\bigg(\frac{4}{|I|+4R}+\big(|I|+4R\big)^{2s}[v_\e]^2_{H^{k+s}(I+(-2R,2R))}\bigg).
       \end{split}
    \]
    However, since  $|I| \leq |A^\eta_\e|+2R(M-1)$ and $[v_\e]_{H^{k+s}(I+(-2R,2R))}$ is bounded by $\Phi^{a_{\scriptstyle \e}}(v_\e, I+(-2R,2R))$ which in turn is uniformly bounded, we have 
    \begin{equation}
    \label{ineq 5.4}
        \|v_\e^{(k)}\|^2_{L^2(I)}\leq C,
    \end{equation}
    which, combined with \eqref{eq:large_intervals_estimate}, gives \eqref{coerc bound k>0}.
    Note that if $I+(-2R,2R)$ is not strictly included in $\left(-\frac{\tau}{\e}, \frac{\tau}{\e}\right)$, then $\Phi^{a_{\scriptstyle \e}}(v_\e, I+(-2R,2R))$ is not necessarily uniformly bounded, but by an application of a classical extension theorem (e.g. Theorem 1.43 of \cite{Leoni_frac_sob_sp}) one can recover that \eqref{ineq 5.4} holds as well. 
    
Now we prove the equiboundedness of $v_\e$. By definition $|v_\e|<1+\eta$ on the large intervals $J_n$.
Let us consider again $I=(b,c)$, a connected component of $\tilde A_\e$. We apply  Lemma \ref{lemma:derivative_bounds} to $(b-R,b)$ to find a point $y\in (b-R,b)$ such that $|v^{(\ell)}_\e(y)|\leq 1$ for every $\ell \in\{1,\dots,k\} $. Then, for every $x\in (b-R,c)$ and $\ell \in \{1,\dots,k\}$, we have
 \begin{equation*}
         |v_\e^{(\ell-1)}(x)|\leq |v^{(\ell-1)}_\e(y)|+ \Big|\int_y^xv_\e^{(\ell)}(t)\,\de x\Big|,
 \end{equation*}
 and consequently
 \begin{align}
    \label{bound k-1} |v_\e^{(k-1)}(x)|&\leq 1 + \sqrt{|I|+R}\,\|v_\e^{(k)}\|_{L^2((b-R,c))},\\
    \label{bound ell} |v_\e^{(\ell-1)}(x)|&\leq 1 + (|I|+R)\|v_\e^{(\ell)}\|_{L^\infty((b-R,c))},\quad \text{for }\ell=2,\dots, k-1,\\
    \label{bound 0} |v_\e(x)|&\leq 1 + \eta + (|I|+R)\|v_\e'\|_{L^\infty((b-R,c))}
 \end{align}
Recalling that the bounds $|I|\leq C_\eta + 2R(M-1)$ are independent of $\e$ and the particular choice of $I\subset \Tilde{A}_\e$, from \eqref{bound k-1} and \eqref{coerc bound k>0}, we conclude that the $v_\e^{(k-1)}$ are bounded on $\Tilde{A}_\e$ independently of $\e$. 
Then, applying iteratively \eqref{bound ell} and at the end \eqref{bound 0}, we conclude that the $v_\e$ are bounded on $\Tilde{A}_\e$ independently of $\e$, and consequently also on the whole $\R$. 

Given that the sequence $(v_\e)_\e$ is equibounded, we have that there exists a $C>0$ independent of $\e$ such that for any bounded interval $K$
\[\|v_\e\|_{L^2(K)} = C\sqrt{|K|},\]
then \eqref{bound H^k+s loc} follows applying multiple times Lemma \ref{ineq:interpolation}.

    \paragraph{Case $k=0$.} In this case, we estimate directly on large intervals
    \begin{equation}
    \label{k=0 bound large inetrv}
        \norm{v_\e-\,\sigma_\e}^2_{L^2(J_n)} = \norm{v_\e-\,\sgn(v_\e)}^2_{L^2(J_n)}\leq C\Phi^{a_{\scriptstyle \e}}(v_\e,J_n),
    \end{equation}
    in place of \eqref{eq:large_intervals_estimate}.
    Then, we consider again $I$, a connected component of $\widetilde{A}_\e$. We notice that, by the Hölder regularity of $v_\e$, for every $x,y\in I$
    $$|v_\e(x)-v_\e(y)|^2 \leq C [v_\e]^2_{H^s(I)}|I|^{2s-1}\leq C|I|^{2s-1}\Phi^a(v_\e,I)\leq CS|I|^{2s-1}.$$ 
    
    By the continuity of $v_\e$,  $\left||v_\e(y)| - 1\right|\leq 2\eta$ for $y$ at boundary of $I$. For such $y$ and for any $x\in I$
    \[
        |v_\e(y)-\sigma_\e(x)| \leq 2(1+\eta).
    \]
    Thus, we have:
    \begin{align}
        \label{k=0 bound infty}\|v_\e-\sigma_\e\|_{L^\infty(I)}&\leq 2(1+\eta)+C|I|^{s-1/2},\\
        \label{k=0 bound 2}\|v_\e-\sigma_\e\|^2_{L^2(I)} &\leq \big(C|I|^{s-\frac{1}{2}}+ 2(1+\eta)\big)^2|I|.
    \end{align}
    Then, recalling that $|I|\leq C_\eta + 2R(M-1)$, we get that \eqref{coerc bound k=0} comes from \eqref{k=0 bound large inetrv} and \eqref{k=0 bound 2}, while the equiboundedness from \eqref{k=0 bound infty} and the fact that $|v_\e-\sigma_\e|<\eta$ outside of $A^\eta_\e$. 
    Then, \eqref{bound H^k+s loc} is a direct consequence.
\end{proof}

We can now prove that the transition energy is strictly positive, showing that the $\Gamma$-limit is not the trivial zero functional.
\begin{Proposition}\label{prop:positivity_of_transition_energies}
    For any homogenization kernel $a$ and $\omega\in \{\pm 1\}$, $$m^\omega(a) >0.$$
\end{Proposition}
\begin{proof}
    Since $m^\omega(a) \geq \min\{1,\alpha_a\}m(1)$, we only have to prove that $m\coloneqq m(1)$ is strictly positive (see also Remark \ref{transition asymmetry} for the drop of the superscript). 
    
    Suppose by contradiction that $m=0$. Then, by Corollary \ref{coroll:min_equivalence}, for every $n>0$ we can find a $\e_n>0$, such that $\e_n\to 0$ as $n \to \infty$, 
    and $v_n$ admissible for $m\left(1,\frac{1}{2\e_n}\right)$ (see Definition \ref{def:transition_energies}) such that
    \begin{equation*}
    \Phi^1\left(v_n, \left(-\frac{1}{2\e_n},\frac{1}{2\e_n}\right)\right) < \frac{1}{n}.
    \end{equation*}
    Take the set $A_{\e_n}^\eta$  in Remark \ref{rmk:equibounded_intervals} and let $I_n=(x_n,y_n) \subset A_{\e_n}^\eta$ be any interval such that $v_n(x_n)< -1+\eta$ and $v_n(y_n)> 1-\eta$. Clearly $|I_n|\leq |A^\eta_{\e_n}|\leq C_\eta$ is bounded uniformly in $n$. 
    
    We consider the translations $w_n(x) \coloneqq v_n(x+x_n)$:  by Theorem \ref{coercivity} for  $J=[0,C_\eta]$
    \[\norm{w_{n}}_{H^{k+s}(J)} = \norm{v_n}_{H^{k+s}(J+x_n)} \leq C.\]
     Then, we can extract a subsequence of $(w_{n_j})_{j}$ such that $w_{n_j}$ converges to some $w$ weakly in $H^{k+s}(J)$ and uniformly. Moreover, we can also suppose that $y_{n_j}-x_{n_j}$ converges to some $z\in J$ by compactness of $J$.
     
     In particular, the uniform convergence  implies that
    \[\int_J W(w)\,\de x = \lim_{j \rightarrow +\infty} \int_J W(w_{n_j})\,\de x \leq \lim_{j\rightarrow +\infty}\Phi^1\left(v_{n_j}, \left(-\frac{1}{2\e_{n_j}},\frac{1}{2\e_{n_j}}\right)\right) = 0.\]
    Thus $w$ is constant and equal to $1$ or $-1$ on $J$, contradicting the fact that, again by uniform convergence,
    \[w(0)=\lim_{j\to +\infty} w_{n_j}(0) \leq -1+\eta,\quad w(z)=\lim_{j\to +\infty} w_{n_j}(y_{n_j}-x_{n_j})\geq 1-\eta.\]
\end{proof}

\section{Limsup inequalities}
Let us recall the notation $S^\pm(u)$ introduced in \eqref{def: S^pm} for the directional jump set of a function  $u \in BV((0,1), \{\pm 1\})$. 
\begin{Theorem}
    For every $u \in BV((0,1), \{\pm 1\})$ there exists a sequence $(u_\e)_{\e}$ in $H^{k+s}((0,1))$ such that $u_\e \to u$ in measure as $\e \to 0$ and, for any $\lambda\in [0,+\infty)$,
    \begin{equation}
    \label{ineq:limsup}
        \limsup_{\substack{\e\to 0 \\ \del/\e\to\lambda}} F_{\e,\del}(u_\e) \leq m^+_\lambda\#S^+(u) + m^-_\lambda\# S^-(u).
    \end{equation}
    If we further suppose $a$ to be continuous, then \eqref{ineq:limsup} is true also for $\lambda=\infty$.
\end{Theorem}
\begin{proof}
    Let us enumerate the elements of the jump-set,
    \[S(u)= \{t_1,\dots, t_M\},  \quad 0\eqqcolon t_0< t_1<t_2<\ldots<t_M<t_{M+1}\coloneqq 1.\] 
    and let  $2\displaystyle\tau=\min_{i\ne j}|t_i-t_j|$ be the minimal distance among jump points or between a jump point and an edge of the interval. We also let
    \[s_j = \sgn\big(u(t_j^+)- u(t_j^-)\big),\quad j=1,\ldots, M\] 
        be the ``direction" of the jumps. L
        
    By Corollary \ref{coroll:min_equivalence}, for any $\lambda\in [0,+\infty]$, \[m^\pm_{\lambda} = \inf_T m^\pm_{\lambda}(T).\]    
   Consequently, for any $N>0$ and $\lambda\in [0,+\infty]$, there exist $T>0$ and a function $v^\pm_\lambda\in H^{k+s}_{\operatorname{loc}}(\R)$ such that $v^\pm_\lambda(x) \equiv \pm\sgn(x)$ for $\abs{x}>T$ and, recall Definition \ref{def:transition_energies},
    \begin{equation}\label{eq:approx_choice}
       \begin{split}
            \Phi^{a(\cdot\,/\lambda)}(v^\pm_\lambda) &\leq m_\lambda^{\pm} + \frac{1}{N}, \quad \text{for }\lambda\in (0,+\infty),\\
             \Phi^{\bar a}(v^\pm_0) &\leq m_0^{\pm} + \frac{1}{N},\\
              \Phi^{a_{\inf}}(v^\pm_\infty) &\leq m_\infty^{\pm} + \frac{1}{N}.\\
       \end{split}
    \end{equation} 
    Moreover, it is also possible to suppose that $T>1$ and, when $\lambda \ne +\infty$, even that $T>\lambda$.
    We note that by construction, $||v^\pm_\lambda||_{L^\infty(\R)} =  ||v^\pm_\lambda||_{L^\infty((-T,T))}<+\infty$ and as well, if $k\geq 1$, $||(v^\pm_\lambda)^{(k)}||_{L^2(\R)} = ||(v^\pm_\lambda)^{(k)}||_{L^2((-T,T))}<+\infty$.

   We shall approximate $u$, both in measure and in energy, by gluing together rescaled versions of the approximate optimal profiles $v_j\coloneqq v^{s_j}_\lambda$ (from now on we consider $\lambda$ fixed and drop the explicit dependence).  For this purpose, we partition $(0,1)$ into sub-intervals symmetric with respect to the jump-points, namely
    \begin{align*}
        I_j &\coloneqq  (b_{j-1}, b_j), \quad j=1,\ldots, M, \\
        b_j &\coloneqq \frac{t_j+t_{j+1}}{2}, \quad j=1,\ldots, M-1, \,b_0\coloneqq 0, \,
        b_{M}\coloneqq 1.
    \end{align*}
     Then, we define the sequences
    \begin{equation}
        \label{eq:recovery_sequence}
        u_\e(x) = \begin{cases}
            \displaystyle\sum_{j=1}^M v_{j}\Big( \frac{\lambda}{ \del}(x-t_j^\delta) \Big)\mathsf{1}_{I_j}(x), &\text{ for }\lambda\in (0,+\infty),\\
            \displaystyle\sum_{j=1}^M v_{j}\Big(\mkern1mu\frac{x-t_j^\delta}{ \e} \mkern1mu\Big)\mathsf{1}_{I_j}(x), &\text{ for }\lambda\in \{0,+\infty\},\\
        \end{cases}
    \end{equation}
    where $\mathsf{1}$'s denote indicator functions and the points $t^\delta_j$ satisfy $|t_j-t_j^\delta|\le\del$, but their precise value will be determined later with respect to $\lambda$.
    
    For $\e$ small enough, precisely such that 
    \begin{equation*}
        \tau_\e\coloneqq \e T+\delta<\tau,
    \end{equation*} $u_\e$ belongs to $H^{k+s}((0,1))$.
    Indeed, for sure $u_\e\in H^{k+s}(I_j)$ for each $j$. Furthermore, as $v_j\in H^{k+s}_{\rm loc}(\R)$, $u_\e\in H^k((0,1))$ by construction.  % and the construction is local, we have
    %\[\|u_\e\|^2_{H^k((0,1))} = \sum_{j=1}^M \|v_j\|^2_{H^k(I_j)}<+\infty.\]
    Then $u_\e\in H^{k+s}((0,1))$ as a consequence of the following claim:
    \begin{equation}
    \label{bound interactions}
        \int_{I_i}\int_{I_j} \frac{|u^{(k)}_\e(y)-u^{(k)}_\e(x)|^2}{|x-y|^{1+2s}} \,\de x\,\de y \leq \frac{C}{(\tau-\tau_\e)^{1+2s}}\big(\|u_\e^{(k)}\|^2_{L^2(I_j)} +  \|u_\e^{(k)}\|^2_{L^2(I_i)}\big).
    \end{equation}
    In order to prove the claim, note that the integrand can be non-zero only if
$|x-y|\geq \tau-\tau_\e$. Indeed, $|x-y|\geq 2\tau$  when  $|i-j|>1$, while, if $i=j+1$, 
\[u^{(k)}_\e(x)-u^{(k)}(y) \equiv 0,\]
for any $x\in\big(t_j^\delta+\e T, b_j)$ and $y\in(b_j,t_{j+1}^\delta-\e T\big)$. Thus, the integrand is non-zero if
\[|x-y| \geq \tau-\e T-\delta\geq \tau-\tau_\e,\]
and consequently 
\begin{align*}
   \int_{I_i}\int_{I_j} \frac{|u^{(k)}_\e(y)-u^{(k)}_\e(x)|^2}{|x-y|^{1+2s}} \,\de x\,\de y&\leq C\int_{I_i}\int_{I_j} \frac{|u^{(k)}_\e(y)-u^{(k)}_\e(x)|^2}{(\tau-\tau_\e)^{1+2s}} \,\de x\,\de y \\
    &\leq \frac{C}{(\tau-\tau_\e)^{1+2s}}\big(\|u_\e^{(k)}\|^2_{L^2(I_j)} +  \|u_\e^{(k)}\|^2_{L^2(I_i)}\big).
\end{align*}

    We note that, for $\e$ small enough, $u_\e$ may differ from $u$ only inside the intervals $t_j+(-\tau_\e,\tau_\e)$
    of total length $2\tau_\e\rightarrow 0$, thus proving convergence in measure. As for the energy, from \eqref{bound interactions}, since $\tau_\e\to 0$, we have:
\begin{equation}\label{eq:energy_difference_recovery}
    \begin{split}
          & F_{\e,\del}(u_\e) - \sum_{j=1}^M F_{\e,\del}(u_\e, I_j) \\
           &\qquad =\sum_{i\neq j} \e^{2(k+s)-1} \int_{I_i}\int_{I_j}a\Big(\frac{x}{\delta},\frac{y}{\delta}\Big)\frac{\big|{u^{(k)}_\e(y)-u^{(k)}_\e(x)}\big|^2}{|y-x|^{1+2s}} \,\de x\,\de y\\
           &\qquad\leq C \e^{2(k+s)-1}\big(\|u_\e^{(k)}\|^2_{L^2(I_j)} +  \|u_\e^{(k)}\|^2_{L^2(I_i)}\big)
    \end{split}
\end{equation}

If $k\geq 1$, we have by construction (see \eqref{eq:recovery_sequence}) that for any $j\in\{1,\dots, M\}$
\[\|u_\e^{(k)}\|^2_{L^2(I_j)} \leq C \e^{1-2k}\|v_j^{(k)}\|^2_{L^2(\R)},\] 
hence, the right-hand side of \eqref{eq:energy_difference_recovery} tends to 0 (as $O(\e^{2s})$). The only subtlety here is that $\e$ must be small enough if $\lambda \in (0,+\infty)$, say $\delta/\e \leq 2\lambda$.
    
If $k=0$, we note that $||u_\e||_{L^\infty(I_j)}\leq ||v_j||_{L^\infty(\R)}<+\infty$, that is, thus,
% we recall again Theorem \ref{coercivity}, which assures us that the $v_\e$, and hence the $u_\e$ are equibounded, thus
\[
\begin{split}
    \int_{I_j}|u_\e(x)|^2\,\de x \leq C |I_j|\leq C.
\end{split}
\]
and, consequently, the right-hand side of \eqref{eq:energy_difference_recovery} tends to zero (as $O(\e^{2s-1})$). These arguments prove that
    \[\limsup_{\substack{\e\to 0 \\ \del/\e\to\lambda}} F_{\e,\delta}(u_\e) \leq \sum_{j=1}^M
   \limsup_{\substack{\e\to 0 \\ \del/\e\to\lambda}}F_{\e,\delta}(u_\e,I_j),\]
   for every $\lambda$. To conclude, from (\ref{eq:approx_choice}) and the arbitrariness of $N$, it is sufficient to show that $$\limsup_{\e \rightarrow 0} F_{\e,\delta}(u_\e, I_j) \leq \Phi^{a_\lambda}(v_j).$$

    \paragraph{Case $\lambda\in (0,+\infty)$.}  We choose  $t_j^\delta = \delta \left\lfloor\frac{t_j}{\delta}\right\rfloor$, so that by the periodicity of $a$ we have $a\Big( \frac{x}{\lambda}+\frac{t^\delta_j}{\del}, \frac{y}{\lambda}+\frac{t^\delta_j}{\del}\Big) = a\left( \frac{x}{\lambda}, \frac{y}{\lambda}\right)$, thus, by a change of variable we get
    \begin{equation*}
      \begin{split}
             F_{\e,\delta}(u_\e, I_j)
             &\le\frac{\del}{\lambda\e}\int_\R W\big(v_j(  x)\big) \,\de  x\\
               &\qquad +\Big(\frac{\del}{\lambda\e}\Big)^{1-2(k+s)}\iint_{\R^2}a\Big( \frac{x}{\lambda}, \frac{y}{\lambda}\Big)\frac{|v_{j}^{(k)}(  y)-v^{(k)}_{j}( x)|^2}{|y-x|^{1+2s}} \,\de x\,\de y\\
                   &=\Phi^{a(\cdot/\lambda)}(v_j)+\mr o_\e(1).
      \end{split}
    \end{equation*}
    %where the last passage follows from the fact that $\delta/\e\to \lambda$.
    \paragraph{Case $\lambda = 0$.}     
    We choose $t_j^\delta$ as in the previous case, and we get
     \begin{equation*}
      \begin{split}
             F_{\e,\delta}&(u_\e, I_j)\\
             %&\leq \int_\R W\big(v_j(  x)\big) \,\de  x +\iint_{\R^2}a\Big( \frac{\e}{\del}x+\frac{t^\delta_j}{\del}, \frac{\e}{\del}y+\frac{t^\delta_j}{\del}\Big)\frac{|v_{j}^{(k)}(  y)-v^{(k)}_{j}( x)|^2}{|y-x|^{1+2s}} \,\de x\,\de y\\ 
              &\leq\int_\R W\big(v_j(  x)\big) \,\de  x +\iint_{\R^2}a\Big( \frac{\e}{\del}x, \frac{\e}{\del}y \Big)\frac{|v_{j}^{(k)}(  y)-v^{(k)}_{j}( x)|^2}{|y-x|^{1+2s}} \,\de x\,\de y\\ 
                   &=\Phi^{\bar a}(v_j)+\mr o_\e(1),
      \end{split}
    \end{equation*}
where the last passage follows from the Riemann--Lebesgue Lemma:
$$a\Big(\frac{\e}{\delta}\cdot\,,\frac{\e}{\delta}\cdot\Big)\overset{*}{\rightharpoonup}  \bar a=\int_0^1\int_0^1a(x,y)\,\de x\,\de y, \quad\text{in } L^\infty.$$

    \paragraph{Case $\lambda = +\infty$.}
    Here we assume $a$ to be continuous, so that  
    \[a_{\inf} = \min_{0\leq t\leq 1} a(t,t)=a(r,r)\] 
    is well defined and attained by some $r\in [0,1)$. We chose
    \[t^\delta_j = \del\Big(\Big\lfloor\frac{t_j}{\delta}-r\Big\rfloor + r\Big),\] so that, taking into account the periodicity of $a$,  
    \[a\Big( \frac{\e}{\del}x+\frac{t^\delta_j}{\del}, \frac{\e}{\del}y+\frac{t^\delta_j}{\del}\Big)= a\left(r+ \frac{\e}{\delta}x, r+ \frac{\e}{\delta}y\right).\]
    Since $\frac{\e}{\delta} \rightarrow 0$ and $a$ is continuous, it follows that pointwise
    \[a\left(r+ \frac{\e}{\delta}x, r+ \frac{\e}{\delta}y\right) \rightarrow a(r,r) = a_{\inf},\]
    and, by the Dominated Convergence Theorem,
    \[a \left(r+ \frac{\e}{\delta}\cdot\, , r+ \frac{\e}{\delta}\cdot\right) \overset{*}{\rightharpoonup}  a_{\inf}\]
    in the weak-$*$ convergence of $L^\infty$. Then, we can repeat the argument of the previous case.
\end{proof}

\section{Liminf inequalities}
\label{Gamma-liminf}
\begin{Theorem}
    Let $u \in BV((0,1), \{\pm 1\})$ and $S^\pm(u)$ be defined as in \eqref{def: S^pm}. Then, for any sequence $(u_\e)_{\e}$ in $H^{k+s}((0,1))$ such that $u_\e \to u$ in measure as $\e \to 0$ and for any $\lambda\in [0,+\infty)$,
    \begin{equation}
    \label{ineq:liminf}
        \liminf_{\substack{\e\to 0\\  \del/\e\to\lambda}} F_{\e,\del}(u_\e) \geq m^+_\lambda\#S^+(u) + m^-_\lambda\# S^-(u).
    \end{equation}
    If we further suppose $a$ to be continuous, then \eqref{ineq:liminf} is true for $\lambda=\infty$ as well.
\end{Theorem}
\begin{proof}
We suppose $\sup_\e F_{\e,\delta}(u_\e) <+\infty$, otherwise the inequality is trivially true. 

     For each jump point $t_j \in S(u)$ we define $I_j = (t_j-\tau, t_j + \tau)$ where $\tau>0$ is chosen small enough to make the intervals $I_j$ disjoint pairwise and from $\{0,1\}$. Then,
   \begin{equation}\label{eq:lower_bound_localization}
       \begin{split}
           \liminf_{\e\to 0} F_{\e,\del}(u_\e) &\geq \liminf_{\e\to 0}\; \sum_j F_{\e,\del}(u_\e, I_j) \\
           &\geq \sum_j \;\liminf_{\e \to 0} F_{\e,\del}(u_\e, I_j).
       \end{split}
   \end{equation}
%Third, by the density of $C^\infty$ in $H^{k+s}$ we can suppose without loss of generality that $u_\e\in C^\infty$. 

We also fix $\eta \in (0,1)$  and consider the local rescalings
 \begin{equation*}
     w_\e^j(x)=u_\e(t_j+\e x),\quad w_\e^j\in H^{k+s}\Big(\Big({-}\frac{\tau}{\e} ,\frac{\tau}{\e} \Big)\Big).
 \end{equation*}
The Remark \ref{rmk:equibounded_intervals} provides us with two constants independent of $\e$, $C_\eta>0$ and
 $M_\eta\in \N$, such that $w_\e^j$ satisfies $||w^j_\e|-1|<\eta$ outside sets $A^{\eta,j}_\e$ of measure less then $C_\eta$. In addition, each $A^{\eta,j}_\e$ is a union of at most $M_\eta$ intervals  and $||w^j_\e|-1\big|>\frac{\eta}{2}$ on  $A^{\eta,j}_\e$. Since $\eta$ will not vary through the proof we will omit explicit dependence on $\eta$ till the end.

We also fix $\theta\in (0,1/3)$  and notice that, since $u_\e\rightarrow u$ in measure, upon considering subsequences and a possibly smaller $\tau$, we can suppose
\begin{equation*}
    \big||w^j_\e|-1\big|\leq \eta, \quad\text{on }\Big({-}\frac{\tau}{\e}, \frac{\tau}{\e}\Big)\setminus\Big({-}\theta\frac{\tau}{\e}, \theta\frac{\tau}{\e}\Big),
\end{equation*}
and \begin{equation*}
    \sgn\big(w^j_\e( \tau/\e )- w^j_\e(- \tau/\e)\big)=\sgn\big(u(t_j+ \tau)- u(t_j- \tau)\big)\eqqcolon s_j.
\end{equation*}
We note that $s_j=\pm 1$ if and only if $t_j\in S^\pm(u)$, and, set 
\[a^{j}_{\e}(x,y)\coloneqq a \Big(\frac{\e}{\delta}(x+t_j),\frac{\e}{\delta}(y+t_j)\Big),\] 
we have
    \begin{equation}\label{eq:F_Phi}
        F_{\e,\del}(u_\e, I_j)=\Phi^{a^{j}_{\scriptstyle\e}}_{\tau/\e}(w_\e^j), 
    \end{equation}
    with the notation for $\Phi^a_T$ introduced after \eqref{eq:def_Phi}. For each $j$,  by   the boundedness in energy of the sequence $u_\e$ and Remark \ref{rmk:multiple_flattening} applied with
   \begin{equation*}
        -b=c\coloneqq \frac{\tau}{\e},\quad -b''=c''\coloneqq \theta\frac{\tau}{\e},\quad -b'=c'\coloneqq 2\theta\frac{\tau}{\e}, 
   \end{equation*}
  for any fixed $N\in\N$, we find a function $v_\e^j \in H^{k+s}_{\operatorname{loc}}(\R)$ such that 
    \begin{equation*}
       v_\e^j(x) = s_j\sgn(x)  , \text{ for }|x| > 2\theta\frac{\tau}{\e} ,
    \end{equation*}
 and
 \begin{equation}\label{eq:lower_bound_flattening}
     \Phi^{a^{j}_{\scriptstyle\e}}_{\tau/ \e}(v_\e^j)\leq \Big(1+\frac{C}{N}+o_\e(1)\Big)\Phi^{a^{j}_{\scriptstyle\e}}_{\tau/ \e}(w^j_\e),
 \end{equation}
 with $o_\e(1)\to 0$ keeping $N$ fixed. Furthermore, by Lemma \ref{lemma:vanishing_tails} 
    \begin{equation}\label{eq:lower_bound_tails}
        \begin{split}
             \Phi^{a^{j}_{\scriptstyle\e}}_{\tau/ \e}(v_\e^j)&\geq \Phi^{a^{j}_{\scriptstyle\e}
             } (v_\e^j)- C\theta^{2s} + o_\e(1).
        \end{split}
    \end{equation}

   \paragraph{Case $\lambda\in (0,+\infty)$.} Recall Definition \ref{def:transition_energies} and Remark \ref{rmk:lambda_continuity}, we observe that 
   \[\Phi^{a^{j}_{\scriptstyle\e}
             } (v_\e^j)\ge m^{ s_j}_{\del/\e, t_j/\e}= m^{ s_j}_{\del/\e} = m_\lambda^{s_j} + o_\e(1).\] 
Hence, combining \eqref{eq:lower_bound_localization}, \eqref{eq:F_Phi}, \eqref{eq:lower_bound_flattening} and \eqref{eq:lower_bound_tails}, we get
    \begin{equation*}
   \begin{split}
     \liminf_{\varepsilon\to 0}   F_{\e,\del}(u_\e, I_j)&\ge \sum_j\liminf_{\varepsilon\to 0}\Big(1+\frac{C}{N}\Big)^{-1}\big(m^{s_j}_{\delta/\e}- C\theta^{2s}\big) \\
     &=  \sum_j\Big(1+\frac{C}{N} \Big)^{-1}(m^{s_j}_\lambda- C\theta^{2s}  ),
   \end{split}
    \end{equation*}
from which one concludes by the arbitrariness of $N$ and $\theta$.

\paragraph{Case $\lambda = 0$.} 
The first step is to further refine the approximating sequences $(v^j_\e)$ and construct new sequences $(\mathfrak{v}^{j}_\e)_\e$ in $H^{k+s}_{\mathrm{loc}}(\R)$, 
depending on a new parameter $T>0$ and large enough,  such that $\mathfrak{v}^{j}_{\e}(x)\equiv \mathfrak{v}^{j, T}_{\e}(x)=s_j\sgn(x)$ for $x>T$, \emph{independent} of $\e$, in addition to
\begin{equation*}
        \begin{split}
             \Phi^{a^{j}_{\scriptstyle\e}}(\mathfrak{v}^{j}_{\e}) &\leq  \Big(1+\frac{C}{N}+o_{T\to +\infty}(1)+o_\e(1)\Big)\Phi^{a^{j}_{\scriptstyle\e}}_{\tau/ \e}(w_\e^j)\\
             &\quad+C\theta^{2s}+ o_\e(1)+ o_{T\to +\infty}(1).
        \end{split}
    \end{equation*}
This improves the construction above, where $v^j_\e\equiv s_j\sgn$  only for $|x|>2\theta\tau/\e$. The fact that $T$ is independent of $\e$ will be essential to obtain strong compactness, which, in turn, will allow us to apply Riemann--Lebesgue Lemma.

\subparagraph{Main step, construction of $\mathfrak{v}^{j, T}_\e$.}
Outside of the same sets $A_\e^\eta$ defined for $w^j_\e$, also  $||v^j_\e|-1|\leq \eta$ by construction. We recall that $A_\e^\eta$ is a disjoint union of at most $M$ intervals and $|A_\e^\eta|\leq C_\eta$, with both $M$ and $C_\eta$ independent of $\e$.

Again by construction, $v^j_\e\equiv s_j\sgn$ for $|x|\geq 2\theta\tau/\e$. Thus there exist two diverging connected components of the set $\{x\colon||v^j_\e(x)|-1|\leq \eta\}$ of the form 
\[\Big({-}\frac{\tau}{\e}, x_-\Big),\quad  \Big(x_+,\frac{\tau}{\e}\Big),\quad \text{with }|x_\pm|\leq 2\theta\frac{\tau}{\e}.\] 
The only obstruction is thus the possible presence of other large intervals $I'\subset (x_-,x_+)$ on which $||v^j_\e|-1|\leq \eta$ (we note that there may be at most $M-1$ of them).

Let $L\geq 1$, and suppose that such an interval exists, say $I'=(b',c')$ with $c'-b'\geq L$. We will construct $\mathfrak{v}^{j}_\e\equiv\mathfrak{v}^{j, T}_\e$ by first flattening $v^j_\e$ on such intervals, then removing a portion of the flattened part. $T$ will then be a function of $L$.

 We apply the flattening Lemma \ref{lemma:flattening} in the version of Remark \ref{rmk:multiple_flattening}, case (2), with 
    $b=-\tau/\e$ and $c= \tau/\e$,  $b'$ and $c'$ as just above, while $b''$ and $c''$ are chosen so that
    \begin{equation*}
    \frac{L}{8} \geq b''-b'\geq \frac{L}{10}, \qquad \frac{L}{8} \geq c'-c''\geq \frac{L}{10}.
\end{equation*}
Then we get $\overline{v}^{j}_\e \in H^{k+s}_{\mathrm{loc}}(\R)$ such that
    \begin{equation}
    \label{eq:tilde_v_properties}
        \begin{split}
       \overline{v}^{j}_\e&\equiv \sgn(v_\e^j),\quad \text{in } (b'',c''),\\
        \overline{v}^{j}_\e &\equiv v^j_\e,\qquad \text{in } (I')^{\mathsf c},\\
            \Phi^{a^{j}_{\scriptstyle\e}}_{\tau/ \e}(\overline{v}^{j}_\e) &\leq \Big(1+\frac{C}{N}+\frac{C_{N}}{L^{2s}}\Big)\Phi^{a^{j}_{\scriptstyle\e}}_{\tau/ \e}(v_\e^j),
        \end{split}
    \end{equation} 
    both $C$ and $C_N$ are independent on $L$. 
    
    Since $\delta/\e \to 0$, for $\e$ small enough we can find $b'''$, $c'''$ such that
\begin{equation*}
    \begin{split}
     \frac{L}{8} \geq b'''-b''\geq \frac{L}{10}, \qquad \frac{L}{8} \geq c''-c'''\geq \frac{L}{10},\quad     \frac{\e}{\delta}(c'''-b''')\in \mathbb Z.
    \end{split}
\end{equation*}
Then, we remove the interval $(b''',c''')$ defining
    \[
    \mathfrak{v}_\e^{j}(x) \coloneqq 
    \begin{cases}
        \overline{v}_\e^{j}(x) & \text{for } x\leq b''',\\
        \overline{v}_\e^{j}(x+c'''-b''') & \text{for } x > b''',
    \end{cases}
    \]
    in particular, in place of $(b',c')$, now $||\mathfrak{v}_\e^{j}|-1|\leq \eta$ on an interval of length 
    \[c'-b'-(c'''-b''')\leq \frac{L}{2}.\]
    Clearly,
    \[\begin{split}
        \Phi^{a^{j}_{\scriptstyle\e}}\big( \mathfrak{v}_\e^{j},(-\infty, b''')\big) &= \Phi^{a^{j}_{\scriptstyle\e}}\big(\overline{v}_\e^{j},(-\infty,b''')\big),\\
        \Phi^{a^{j}_{\scriptstyle\e}}\big( \mathfrak{v}_\e^{j},(b''',+\infty)\big) &= \Phi^{a^{j}_{\scriptstyle\e}}\big(\overline{v}_\e^{j},(c''',+\infty)\big),
    \end{split}\]
    where the second equality follows from $ \e(c'''-b''')/\delta \in \mathbb Z$, since $a(x,y)$ is $1$-periodic. We are left with the contribution corresponding to the non-local part of $\Phi^{a^{j}_{\scriptstyle\e}}\big( \mathfrak{v}_\e^{j})$ for  $(x,y)\in (-\infty, b''')\times(b''',+\infty)$ (or vice versa, but the reasoning is analogous).
    
    If $x,y$ are both close to $b'''$, at distance less then $L/10$ to be precise, $\mathfrak{v}_\e^{j}(x)=\mathfrak{v}_\e^{j}(y)\in\{\pm 1\}$, and thus the corresponding contribution is null. 
    
    Otherwise $|x-y|\geq L/10$ and, if $k\geq 1$,
    \[\begin{split}
       & \int_{-\infty}^{b'''}\int_{b'''}^{+\infty} a_\e^j(x,y) \frac{\abs{(\mathfrak{v}_\e^j)^{(k)}(x)-(\mathfrak{v}_\e^j)^{(k)}(y)}^2}{|x-y|^{1+2s}}\,\de x\,\de y\\
       &=\int_{-\infty}^{b'''}\int_{b'''}^{+\infty}\mathsf{1}_{\{|x-y|\geq L/5\}}(x,y)\, a_\e^j(x,y) \frac{\abs{(\mathfrak{v}_\e^j)^{(k)}(x)-(\mathfrak{v}_\e^j)^{(k)}(y)}^2}{|x-y|^{1+2s}}\,\de x\,\de y\\
       &\leq C \int_{\R}\int_{L/10}^{+\infty} \frac{\abs{(\mathfrak{v}_\e^j)^{(k)}(x)}^2}{\xi^{1+2s}}\,\de x\,\de \xi \\
       &\leq \frac{C}{L^{2s}}\|(\mathfrak{v}_\e^j)^{(k)}\|^2_{L^2(\R)} \\
       &\leq \frac{C}{L^{2s} }\|(\overline{v}^j_\e)^{(k)}\|^2_{L^2(\R)},
    \end{split}\]
    which is bounded by $C L^{-2s}$ thanks to Theorem \ref{coercivity}.
    
    {On the other hand, if $k=0$, we let $\mathfrak{s}^j_\e$ to be the sign regularizations constructed in Lemma \ref{lemma:sign_v_regularization} in relation to $\mathfrak{v}^j_\e$. By Lemma \ref{sigma Hs bound}, 
    \[
       \int_{-\infty}^{b'''}\int_{b'''}^{+\infty} \frac{|\mathfrak{s}^j_\e(x)-\mathfrak{s}_\e^j(y)|^2}{\abs{x-y}^{1+2s}}\,\de x\,\de y
        \leq C\int_{L/10}^{+\infty}\frac{\xi}{\xi^{1+2s}}\,\de \xi \leq \frac{C}{L^{2s-1}},
    \]
    from which,
    \[
        \int_{-\infty}^{b'''}\int_{b'''}^{+\infty} a_\e^j(x,y) \frac{\abs{\mathfrak{v}_\e^j(x)-\mathfrak{v}_\e^j(y)}^2}{\abs{x-y}^{1+2s}}\,\de x\,\de y\leq \frac{C}{L^{2s}}\norm{\mathfrak{v}_\e^j-\mathfrak{s}_\e^j}^2_{L^2(\R)} + \frac{C}{L^{2s-1}}.
    \]
    }
    Again, we conclude thanks to Theorem \ref{coercivity}, since
\[\norm{\mathfrak{v}_\e^j-\mathfrak{s}_\e^j}^2_{L^2(\R)} =\norm{\overline{v}_\e^j-\overline{\sigma}_\e^j}^2_{L^2(\R)},\]
  being $\overline{\sigma}^j_\e$ the sign regularization provided by Lemma \ref{lemma:sign_v_regularization} in relation to $\overline{v}^j_\e$.  
  
    In conclusion, 
 \begin{equation}
     \label{eq:energy_v_slashed}
     \begin{split}
            \Phi^{a^{j}_{\scriptstyle\e}}(\mathfrak{v}_\e^j ) &\leq \Phi^{a_{\scriptstyle\e}^j}\left(\overline{v}_\e^j\right) + o_{L\to \infty}(1)\\
        &\leq \Phi^{a_{\scriptstyle\e}^j}_{\tau/\e}\left(\overline{v}_\e^j\right) + C\theta^{2s}+o_\e(1)+o_{L\to \infty}(1)\\
        &\leq \Big(1+\frac{C}{N}+o_\e(1)\Big)\Big(1+\frac{C}{N}+\frac{C}{L^2s}\Big)F_{\e,\delta}(u_\e,I_k) \\
        &\qquad+ C\theta^{2s}+o_\e(1)+o_{L\to \infty}(1),
     \end{split}
 \end{equation}
 where we applied, in order, Lemma \ref{lemma:vanishing_tails} on  $\overline{v}_\e^j$, \eqref{eq:tilde_v_properties}, \eqref{eq:lower_bound_flattening} and \eqref{eq:F_Phi}.
    %since $\overline{v}_\e^j$ still satisfy the same properties of $v^j_\e$ needed to apply Lemma \ref{lemma:vanishing_tails} on the tails, and recalling \eqref{eq:F_Phi},\eqref{eq:lower_bound_flattening},\eqref{eq:tilde_v_properties}.

    Repeating this procedure of flattening and removal for all the (at most $M-1$) problematic large intervals, we construct a new sequence, which we still denote with $\mathfrak{v}_\e^{j}$, such that all interior intervals on which $||\mathfrak{v}_\e^{j}|-1|\leq \eta$ have length at most $L$. 
    Notice that the two extremal connected components of the set $\{x\colon||\mathfrak v^j_\e(x)|-1|\leq \eta\}$, 
\[\Big({-}\frac{\tau}{\e}, x_-\Big),\quad  \Big(x_+,\frac{\tau}{\e}\Big),\] 
now satisfy $x_+-x_-\leq (M-1)L+ |A_\e^\eta|\eqqcolon 2T$, so that, upon a possible translation and small abuse of notation, $\mathfrak v^j_\e(x)=s_j \sgn(x)$ for $|x|\geq T$, as desired.

\subparagraph{Final step.}   By Theorem \ref{coercivity},
   the sequence $(\mathfrak v^j_\e)_\e$ is bounded in $H^{k+s}((-2T,2T))$ and uniformly.
    Therefore, we can extract a subsequence converging to some $\mathfrak v^j\in H^{k+s}((-2T,2T))$, 
    \begin{itemize}
        \item weakly in $H^{k+s}((-2T,2T))$ by Banach-Alaoglu;
        \item strongly in $H^k((-2T,2T))$ by the compact embeddings of Sobolev spaces (recall \ref{thm:sobolev_embd});
         \item uniformly by Morrey's embeddings (recall again \ref{thm:sobolev_embd}) and Arzelà-Ascoli.
    \end{itemize}
    We notice that by the uniform convergence $\mathfrak v^j(x)=s_j \sgn(x)$ for $|x|>T$, therefore we can extend it on the rest of $\R$ as $s_j\sgn$.
    Lastly, we consider a new parameter $R>0$ and define the diagonal stripe
    \begin{equation}
    \label{def:close to diagonal set}
        D_R \coloneqq \big\{(x,y)\in\R^2 \,\big|\,|x-y|<R\big\}.
    \end{equation}
    We also denote
    $$f_\e(x,y) \coloneqq \frac{\abs{(\mathfrak{ v}_\e^j)^{(k)}(x) - (\mathfrak{v}_\e^j)^{(k)}(y)}^2}{\abs{x-y}^{1+2s}}, \quad f(x,y) \coloneqq \frac{\abs{(\mathfrak{v}^j)^{(k)}(x)-(\mathfrak v^j)^{(k)}(y)}^2}{\abs{x-y}^{1+2s}}.$$ 
    By the previous observations on the convergence of the $\mathfrak{ v^j_\e}$, $f_\e$ converge to $f$ strongly in $L^1((-T,T)^2\setminus D_R)$, and thus in $L^1(\R^2\setminus D_R)$, since $\mathfrak{ v^j}= \mathfrak{v^j_\e}\equiv \pm 1$ on $\R\setminus(-T,T)$.
    
   We can finally apply the Riemann--Lebesgue Lemma and get:
    \begin{align*}
        &\lim_{\e\rightarrow0} \iint_{\R^2\setminus D_R}a^j_\e(x,y) f_\e(x,y) \,\de x\,\de y =\bar a\iint_{\R^2\setminus D_R}f(x,y) \,\de x\,\de y.
    \end{align*}
    As a consequence,
    \begin{align*}
        \liminf_\e \Phi^{a^{j}_{\scriptstyle\e}}(\mathfrak{v}_\e^j) 
        &\geq \int_\R W(\mathfrak v^j)\,\de x + \bar a\iint_{\R^2\setminus D_R} f(x,y) \,\de x\,\de y,
        \end{align*}
    where the limit of the double-well energies exists due to the uniform convergence of the $\mathfrak{v}^j_\e$. 
    % Now, since $v^j\in H^{k+s}(-T-1,T+1)$ and $\abs{v^j(x)} = 1$ for $|x| >T$, then, with an argument similar to Lemma \ref{lemma:vanishing_tails}, $f(x,y)$ is actually integrable on $\R^2$. 
    Thus, sending $R$ to $0$, 
    \begin{equation}
    \label{ineq Phi}
        \liminf_\e \Phi^{a^{j}_{\scriptstyle\e}}(\mathfrak v^j_\e) \geq \Phi^{\bar a}(\mathfrak v^j) \geq m_0^{s_j}.
    \end{equation}
    Then, combining \eqref{ineq Phi} with \eqref{eq:energy_v_slashed} and \eqref{eq:lower_bound_localization} we conclude letting first $T$ and then $N$ go to $+\infty$, and lastly $\theta \to 0$.
\smallskip

\paragraph{Case $\lambda = +\infty$.} 
We let $R>0$ and consider again the diagonal stripe $D_R$ defined in \eqref{def:close to diagonal set}. Since $a$ is continuous and periodic, therefore uniformly continuous, and since $\frac{\delta}{\e}\rightarrow +\infty$, we have that, for every $(x,y)\in D_R$
\begin{equation*}
    \begin{split}
        a^j_\e(x,y) &= a\Big(\frac{t_j}{\delta}+ \frac{\e x}{\delta}, \frac{t_j}{\delta}+ \frac{\e y}{\delta} \Big) \\
        &\geq a\Big(\frac{t_j}{\delta}+ \frac{\e x}{\delta}, \frac{t_j}{\delta}+ \frac{\e x}{\delta} \Big)-\frac{1}{N} \\
        &\geq a_{\inf} -\frac{1}{N},
    \end{split}
\end{equation*}
   for all $\e$ small enough, depending on $R$ and $N$.
   
    Since $[v^j_\e]_{H^{k+s}((-\tau/\e,\,\tau/\e))}\leq \frac{1}{\alpha_a}\Phi^{a^{j}_{\scriptstyle\e}}_{\tau/ \e}(v_\e^j)$ is bounded in $\e$, for all $\e$ small enough as just above,
    \begin{align}
        \nonumber &\iint_{(-{\tau}/{\e},{\tau}/{\e})^2\cap D_R} a^j_\e(x,y)\frac{\abs{(v_\e^j)^{(k)}(x) - (v_\e^j)^{(k)}(y)}^2}{\abs{x-y}^{1+2s}}\, \de x\,\de y\\
        &\quad \label{stima vicino}
        \geq a_{\inf}\iint_{(-{\tau}/{\e},{\tau}/{\e})^2\cap D_R}  \frac{\abs{(v_\e^j)^{(k)}(x) - (v_\e^j)^{(k)}(y)}^2}{\abs{x-y}^{1+2s}}\, \de x\,\de y \;- \;\frac{C}{N}.
    \end{align}
    On the other hand, {for $k\geq 1$,} 
\begin{equation}
    \begin{split}\label{eq:stima_lontano_k1}
          &\iint_{(-{\tau}/{\e},\,{\tau}/{\e})^2\setminus D_R} \frac{\abs{(v_\e^j)^{(k)}(x) - (v_\e^j)^{(k)}(y)}^2}{\abs{x-y}^{1+2s}}\, \de x\,\de y  \\
        &\quad \leq C \int_R^{+\infty}\frac{1}{\xi^{1+2s}}\,\de \xi \int_{-\tau/\e}^{\tau/\e}\abs{(v_\e^j)^{(k)}(x)}^2\,\de x \\
        &\quad = \frac{C}{R^{2s}}\|(v_\e^j)^{(k)}\|^2_{L^2((-\tau/\e,\,\tau/\e))}\\
        &\quad \leq \frac{C}{R^{2s}}.
    \end{split}
\end{equation}
    where we used Theorem \ref{coercivity} in the last passage.
    
    {While, for $k=0$, we consider the sign regularizations $\sigma^j_\e$ constructed in Lemma \ref{lemma:sign_v_regularization} in relation to $v^j_\e$. We compute again

    \begin{align}
         \notag&\iint_{(-{\tau}/{\e},{\tau}/{\e})^2\setminus D_R} \frac{\abs{v_\e^j(x) - v_\e^j(y)}^2}{\abs{x-y}^{1+2s}}\, \de x\,\de y  \\
        \notag&\quad = \frac{C}{R^{2s}}\|{v_\e^j-\sigma_\e^j}\|^2_{L^2((-\tau/\e,\,\tau/\e ))}  + C\int_R^{+\infty}\int_{-\tau/\e}^{\tau/\e} \frac{\abs{\sigma_\e^j(x+\xi) - \sigma_\e^j(x)}^2}{\xi^{1+2s}}\, \de x\,\de \xi\\
        \notag&\quad \leq  \frac{C}{R^{2s}}+C\int_R^{+\infty} \frac{|\xi|}{|\xi|^{2s+1}}\\
        \label{eq:stima_lontano_k0}&\quad \leq \frac{C}{R^{2s}}+\frac{C}{R^{2s-1}},
    \end{align}
    where we used Theorem \ref{coercivity} again to bound the first term and Lemma \ref{sigma Hs bound} for the second one.
    } 
    
    Thus, putting together \eqref{eq:stima_lontano_k1} and \eqref{eq:stima_lontano_k0}, and since $|a^j_\e(x,y) - a_{\inf}|\leq \beta_a - \alpha_a$, we get, for every $k$,
\begin{equation}
    \begin{split}
           & \iint_{(-{\tau}/{\e},{\tau}/{\e})^2\setminus D_R} a^j_\e(x,y)\frac{\abs{(v_\e^j)^{(k)}(x) - (v_\e^j)^{(k)}(y)}^2}{\abs{x-y}^{1+2s}}\, \de x\,\de y\\
        &\label{stima lontano}
        \quad \geq a_{\inf}\iint_{(-{\tau}/{\e},{\tau}/{\e})^2\setminus D_R} \frac{\abs{(v_\e^j)^{(k)}(x) - (v_\e^j)^{(k)}(y)}^2}{\abs{x-y}^{1+2s}}\, \de x\,\de y - o_{R\to \infty}(1)
    \end{split}
\end{equation}

    In conclusion, adding \eqref{stima vicino} and \eqref{stima lontano}, we have
    \[ \Phi^{a^{j}_{\scriptstyle\e}}_{\tau/ \e}(v_\e^j) \geq  \Phi^{a_{\inf}}_{\tau/ \e}(v_\e^j) -\frac{C}{N}-o_{R\to \infty}(1).\]
    Therefore,  applying Lemma \ref{lemma:vanishing_tails} on $\Phi^{a_{\inf}}_{\tau/ \e}$, as for \eqref{eq:lower_bound_tails}, we get
    \begin{equation*}
        \liminf_\e \Phi^{a^{j}_{\scriptstyle\e}}_{\tau/ \e}(v_\e^j) \geq m^{s_j}_\infty -C\theta^{2s} - \frac{C}{N} -o_{R\to \infty}(1),
    \end{equation*}
    which, recalled \eqref{eq:lower_bound_flattening}, \eqref{eq:lower_bound_tails} and \eqref{eq:lower_bound_localization}, allows us to conlcude by the arbitrariness of $R$, $N$, and $\theta$.
\end{proof}

\section{Acknowledgements}
The authors wish to thank Prof. Andrea Braides for having proposed the problem within the course ``Singular perturbations in fractional Sobolev spaces", held at SISSA in 2025, and for the help he provided, suggesting some of the ideas. F.C. is a member of GNFM, and S.S and E.V. are members of GNAMPA of INdAM.

The authors state that there is no conflict of interest.

%\nocite{}
\printbibliography
%\bibliographystyle{siam}
%\bibliography{biblio}

\end{document}